\documentclass[11pt,reqno,tbtags]{amsart}
\usepackage{amssymb}
\usepackage{natbib}
\bibpunct[, ]{[}{]}{;}{n}{,}{,}

\numberwithin{equation}{section}
\allowdisplaybreaks

\newtheorem{theorem}{Theorem}[section]
\newtheorem{lemma}[theorem]{Lemma}

\theoremstyle{definition}
\newtheorem{example}[theorem]{Example}
\newtheorem{definition}[theorem]{Definition}
\newtheorem{problem}[theorem]{Problem}
\newtheorem{remark}[theorem]{Remark}

\newtheorem*{ack}{Acknowledgement}

\theoremstyle{remark}

\newenvironment{romenumerate}{\begin{enumerate}% gives (i), (ii) etc.
 }{\end{enumerate}}

\newcounter{oldenumi}
\newenvironment{romenumerateq}% continues numbering from previous romenumerate
{\setcounter{oldenumi}{\value{enumi}}
\begin{romenumerate} \setcounter{enumi}{\value{oldenumi}}}
{\end{romenumerate}}

\newcounter{thmenumerate}
\newenvironment{thmenumerate}
{\setcounter{thmenumerate}{0}%
 \def\item{\par% \ifnum\thethmenumerate=0\else\par\fi %\noindent\fi
 \refstepcounter{thmenumerate}\textup{(\roman{thmenumerate})\enspace}}
}
{}

\newcounter{xenumerate}   %no left indentation; thus wider lines

\newcommand{\refT}[1]{Theorem~\ref{#1}}

\newcommand{\refL}[1]{Lemma~\ref{#1}}
\newcommand{\refR}[1]{Remark~\ref{#1}}
\newcommand{\refS}[1]{Section~\ref{#1}}
\newcommand{\refD}[1]{Definition~\ref{#1}}
\newcommand{\refP}[1]{Problem~\ref{#1}}
\newcommand{\refE}[1]{Example~\ref{#1}}

\newcommand{\refand}[2]{\ref{#1} and~\ref{#2}}

\newcommand\marginal[1]{\marginpar{\raggedright\parindent=0pt\tiny #1}}

\begingroup
  \divide\count255 by 60
  \multiply\count255 by -60
  \advance\count255 by \time
  \ifnum \count255 < 10 \xdef\klockan{\the\count1.0\the\count255}
  \else\xdef\klockan{\the\count1.\the\count255}\fi
\endgroup

\newcommand\nopf{\qed}   % for theorem without proof
\newcommand\set[1]{\ensuremath{\{#1\}}}

\newcommand\xpar[1]{(#1)}
\newcommand\bigpar[1]{\bigl(#1\bigr)}

\newcommand\biggpar[1]{\biggl(#1\biggr)}

\newcommand\bigabs[1]{\bigl|#1\bigr|}
\newcommand\Bigabs[1]{\Bigl|#1\Bigr|}

\def\rompar(#1){\textup(#1\textup)}    % usage: \rompar(...)
\newcommand\xfrac[2]{#1/#2}

\def\xexp(#1){e^{#1}}

\newcommand\floor[1]{\lfloor#1\rfloor}

\newcommand\ntoo{\ensuremath{{n\to\infty}}}

\newcommand\norm[1]{\|#1\|}

\newcommand\iid{i.i.d.\spacefactor=1000}    
\newcommand\ie{i.e.\spacefactor=1000}
\newcommand\eg{e.g.\spacefactor=1000}

\newcommand\cf{cf.\spacefactor=1000}
\newcommand{\as}{a.s.\spacefactor=1000}
\newcommand{\aex}{a.e.\spacefactor=1000}
\newcommand{\tend}{\longrightarrow}
\newcommand\dto{\overset{\mathrm{d}}{\tend}}
\newcommand\pto{\overset{\mathrm{p}}{\tend}}
\newcommand\asto{\overset{\mathrm{a.s.}}{\tend}}
\newcommand\towx{\overset{\mathrm{w}\text{-}*}{\tend}}
\newcommand\eqd{\overset{\mathrm{d}}{=}}

\newcommand\bbR{\mathbb R}

\newcommand\bbN{\mathbb N}  %{1,2,...}

\newcounter{CC}
\newcounter{cc}
\newcommand\E{\operatorname{\mathbb E{}}}
\renewcommand\P{\operatorname{\mathbb P{}}}

\newcommand\ga{\alpha}
\newcommand\gb{\beta}

\newcommand\gD{\Delta}
\newcommand\gf{\varphi}
\newcommand\gam{\gamma}
\newcommand\gG{\Gamma}

\newcommand\gl{\lambda}

\newcommand\gs{\sigma}

\newcommand\eps{\varepsilon}

\newcommand\cB{\mathcal B}

\newcommand\cD{\mathcal D}
\newcommand\cE{\mathcal E}
\newcommand\cF{\mathcal F}

\newcommand\cK{\mathcal K}

\newcommand\cP{\mathcal P}

\newcommand\cS{{\mathcal S}}

\newcommand\cW{\mathcal W}

\newcommand\ett[1]{\boldsymbol1[#1]} 
\newcommand\etta{\boldsymbol1} 
\def\[#1]{[\![#1]\!]}

\newcommand\qw{^{-1}}
\newcommand\qww{^{-2}}

\renewcommand{\=}{:=}

\newcommand\oi{[0,1]}
\newcommand\oiq{(0,1)}

\newcommand\setoi{\set{0,1}}

\newcommand\dd{\,\textup{d}}

\newcommand\gnp{\ensuremath{G(n,p)}}

\newcommand{\tinj}{t_{\mathrm{inj}}}
\newcommand{\tind}{t_{\mathrm{ind}}}
\newcommand{\taup}{\tau^+}
\newcommand{\tauinj}{\tau_{\mathrm{inj}}}
\newcommand{\tauind}{\tau_{\mathrm{ind}}}
\newcommand{\ptau}{\hat\tau}
\newcommand{\ptaup}{\ptau^+}
\newcommand{\ptauinj}{\ptau_{\mathrm{inj}}}
\newcommand{\ptauind}{\ptau_{\mathrm{ind}}}
\newcommand{\rest}[1]{|_{[#1]}}

\newcommand{\restx}[1]{|_{#1}}
\newcommand{\DU}{\cD}
\newcommand{\DUx}[1]{\cD_{#1}}
\newcommand{\DUn}{\DUx{n}}
\newcommand{\DL}{\cD^L}
\newcommand{\DLx}[1]{\cD^L_{#1}}
\newcommand{\DLn}{\DLx{n}}

\newcommand{\cdq}{\overline{\cD}}
\newcommand{\cdd}{\cD^+}
\newcommand{\cdoo}{\DUx{\infty}}

\newcommand{\oii}{\oi^{2}}
\newcommand{\oid}{\oi^{\cD}}
\newcommand{\oidd}{\oi^{\cdd}}
\newcommand{\cw}{\cW}

\newcommand{\cwoi}{\cW(\oi)}
\newcommand{\cwsss}{\cW(\sss)}
\newcommand{\bcw}{\overline\cW}
\newcommand{\bcwcut}{(\bcw,\dcut)}
\newcommand{\bcwp}{\overline\cW_{\mathsf P}}

\newcommand{\gbwx}[1]{G(#1,\bW)}
\newcommand{\gbwoo}{\gbwx\infty}
\newcommand{\gbwn}{\gbwx{n}}

\newcommand{\ubw}{\Gamma_{\bW}}

\newcommand{\sn}{\mathfrak S_n}

\newcommand{\xio}{\xi_\emptyset}
\newcommand{\xii}{\xi_i}
\newcommand{\xij}{\xi_j}
\newcommand{\xiij}{\xi_{ij}}
\newcommand{\xiji}{\xi_{ji}}
\newcommand{\iij}{I_{ij}}
\newcommand{\uoi}{U(0,1)}

\newcommand{\wab}{W_{\ga\gb}}
\newcommand{\wo}{w}%W_{0}}
\newcommand{\bW}{\mathbf{W}}
\newcommand{\oW}{\overline{W}}
\newcommand{\tW}{\widetilde{W}}

\newcommand{\WW}{\cW_5}

\newcommand{\exch}{exchangeable}
\newcommand\xn{\ensuremath{[n]}}

\newcommand{\xk}{[k]}
\newcommand{\vvk}{v_1,\dots,v_k}
\newcommand{\vvki}{v'_1,\dots,v'_k}
\newcommand{\PU}{\cP}
\newcommand{\PUx}[1]{\cP_{#1}}
\newcommand{\PUn}{\PUx{n}}
\newcommand{\PL}{\cP^L}
\newcommand{\PLx}[1]{\cP^L_{#1}}
\newcommand{\PLn}{\PLx{n}}

\newcommand{\cpq}{\overline{\cP}}
\newcommand{\cpp}{\cP^+}
\newcommand{\cpoo}{\PUx{\infty}}
\newcommand{\cploo}{\PLx{\bbN}}
\newcommand{\oip}{\oi^{\cP}}
\newcommand{\oipp}{\oi^{\cpp}}

\newcommand{\Borgsetal}{Borgs, Chayes, Lov\'asz, S\'os and Vesztergombi}
\newcommand\sC{\mathsf{C}}
\newcommand\sD{\mathsf{D}}
\newcommand\sP{\mathsf{P}}
\newcommand\bbNx{\bbN\cup\set\infty}
\newcommand\ops{ordered probability space}
\newcommand\sfmuu{\ensuremath{(\cS,\cF,\allowbreak\mu,<)}}
\newcommand\sfmuux{\ensuremath{(\cS,\cF,\allowbreak\mu,\prec)}}
\newcommand\sfmu{\ensuremath{(\cS,\cF,\mu)}}

\newcommand\oibglx{\ensuremath{(\oi,\cB,\allowbreak\gl,\prec)}}
\newcommand\oibgl{\ensuremath{(\oi,\cB,\allowbreak\gl,<)}}
\newcommand\oib{\ensuremath{(\oi,\cB,\gl)}}
\newcommand\bbnn{\bbN\cup\set\infty}
\newcommand\csq{\cS^{|Q|}}
\newcommand\csf{\cS^{|F|}}
\newcommand\xQ{{|Q|}}
\newcommand\xF{{|F|}}
\newcommand\piw{\Pi_W}
\newcommand\pip{\Pi_P}
\newcommand\piwp{\Pi_{W_P}}
\newcommand\piwi{\Pi_{W_1}}
\newcommand\piwii{\Pi_{W_2}}

\newcommand\pit{\Pi_T}
\newcommand\pie{\Pi_0}
\newcommand\nn{[n]}
\newcommand\oo{\bbN}
\newcommand\pnw{\ensuremath{P(n,W)}}
\newcommand\poow{\ensuremath{P(\infty,W)}}
\newcommand\pnp{P(n,\Pi)}
\newcommand\pnpi{\ensuremath{P(n,\Pi)}}
\newcommand\poopi{\ensuremath{P(\infty,\Pi)}}

\newcommand\poopiw{\ensuremath{P(\infty,\piw)}}

\newcommand\poox[1]{\ensuremath{P(\infty,#1)}}
\newcommand\sss{\cS}
\newcommand\lm{Lebesgue measure}
\newcommand\oilm{$\oi$ with Lebesgue measure}
\newcommand\lp{Lebesgue point}
\newcommand\hP{\widehat P}
\newcommand\hpn{\widehat {P_n}}
\newcommand\mup{\mu_p}
\newcommand\pipp{\Pi_{(p)}}
\newcommand\sssx{\sss^*}
\newcommand\np[1]{|#1|_+}
\newcommand\erip{exchangeable random infinite poset}

\newcommand\oiqq{\oiq^2}
\newcommand\xex{X_x^\eps}
\newcommand\xey{X_y^\eps}
\newcommand\xez{X_z^\eps}
\newcommand\gr{^\circ}
\newcommand\txi{\tilde \xi}
\newcommand\teta{\tilde \eta}
\newcommand\precy{\prec^*}
\newcommand\precpnw{\prec_{\pnw}}
\newcommand\ps{probability space}
\newcommand\sssq{\sss^2}
\newcommand\mpx{measure preserving}

\newcommand\cn[1]{\norm{#1}\cut}

\newcommand\cntwo[1]{\norm{#1}_{\square,2}}
\newcommand\cnone[1]{\norm{#1}_{\square,1}}
\newcommand\cut{_{\square}}

\newcommand\dcut{{\delta_{\square}}}

\newcommand\on[1]{\norm{#1}_{L^1}}
\renewcommand\sn[1]{\norm{#1}_{\infty}}
\newcommand{\Lovasz}{Lov\'asz}

\newcommand\REM[1]{{\raggedright\texttt{[#1]}\par\marginal{XXX}}}

\newcommand\citex[1]{\texttt{[#1]}}
\newcommand\refx[1]{\texttt{#1}}

\newcommand\urladdrx[1]{{\urladdr{\def~{{\tiny$\sim$}}#1}}}

\begin{document}
\title%[Poset limits and exchangeable random posets]
{Poset limits and exchangeable random posets}

\date{January 25, 2009}  % (typeset \today{} \klockan)} %; revised ...

%\author{Svante Janson}
\address{Department of Mathematics, Uppsala University, PO Box 480,
SE-751~06 Uppsala, Sweden}
\email{svante.janson@math.uu.se}
\urladdrx{http://www.math.uu.se/~svante/}

%\keywords{<keywords>}
%{60C05 (68P10,68W40)} %%{Primary: <subject>; Secondary: <subject>}

\begin{abstract} 
We develop a theory of limits of finite posets in close analogy to the
recent theory of graph limits. In particular, we study representations
of the limits by functions of two variables on a probability space,
and connections to exchangeable random infinite posets.
\end{abstract}

\maketitle

\section{Introduction and main results}\label{Sintro}

A deep 
theory of limit objects of (finite) graphs has in recent years been created
by
\citet{LSz} and
\citet{BCLSV1,BCLSV2}, and further developed in a series of papers by
these and other 
authors. 
It is shown by \citet{SJ209}
that the theory is closely connected with
the Aldous--Hoover theory of representations of exchangeable arrays of
random variables, further developed and described in detail by
\citet{Kallenberg:exch}; the connection is through 
exchangeable random infinite graphs.
(See also \citet{Tao} and \citet{Austin}.)

The basic ideas of the graph limit theory extend to other structures
too; note that the Aldous--Hoover theory 
as stated by \citet{Kallenberg:exch} 
includes both multi-dimensional arrays (corresponding to hypergraphs) 
and some different symmetry conditions (or lack thereof).
For bipartite graphs and digraphs (\ie, directed graphs), some details
are given by 
\citet{SJ209}. 
For hypergraphs, an extension is given by
\citet{ElekSz}; see also
\cite{SJ209} (where no details are given)
and \citet{Tao} and \citet{Austin}.

It seems possible that some future version of the theory will be
formulated in a general way that includes all these cases as well as
others. While waiting for such a theory, it is interesting to study
further structures.
In the present paper, we develop a theory for 
limits of finite \emph{posets} (\ie, partially ordered sets). 

The theory for posets can be developed in analogy with the theory for
graph limits, but it can also be obtained as a special case of the
theory for digraphs. We will in this paper use both views.

In this paper, all posets (and graphs) are assumed to be non-empty. They
are usually finite, but we will sometimes use infinite posets as well.
If $(P,<)$ is a poset, we call $P$ its \emph{ground set}; we also say
that $(P,<)$ is a poset {on} $P$.
For simplicity, we often use the same notation for a poset and
its ground set when there is no danger of confusion. 
Sometimes we write $<_P$ for the partial order and $P\gr$ for
the ground set of a poset $P$.
We let $\cP$ denote the set of unlabelled finite posets.
(For this and other definitions, see also Sections
\ref{Sprel}--\ref{Slim} where more details are given.)

We may regard a poset $(P,<)$ as a digraph, with vertex set $P$ and a
directed edge $i\to j$ if and only if $i<j$ for all $i,j\in P$. (In
particular, the digraph is loopless.) 
The poset and the digraph determine each other uniquely, so
we may identify a poset with the corresponding digraph, but
note that not every digraph is a poset.
Hence, we can regard $\cP$ as a subset of the set
$\cD$ of unlabelled finite digraphs.
A simple characterizations of the digraphs that are posets is given in
\refL{L1}.

A \emph{poset homomorphism} $Q\to P$ is a map
$\gf:Q\gr\to P\gr$ between the ground sets such that
$x<_Qy \implies \gf(x)<_P\gf(y)$.
We say that $Q$ is a \emph{subposet} of $P$, and write
$Q\subseteq P$, if $Q\gr\subseteq P\gr$ and 
$x<_Qy \implies x<_Py$, \ie, if the identity map $Q\to P$
is a poset homomorphism.
We say that $Q$ is an \emph{induced subposet} of
$P$ if further 
$x<_Qy \iff x<_Py$ for all $x,y\in Q\gr$.
If $P$ is a poset and $A$ is a subset of its ground set $P\gr$, then $P|_A$
denotes the restriction of $P$ to $A$, \ie, $A$ with the order
$<_P$ inherited from $P$. Thus, $Q$ is an induced subposet of
$P$ if and only if $Q$ equals $P|_A$ for some
(non-empty) $A\subseteq P\gr$.
Note that these definitions agree with the corresponding
definitions for digraphs, so we may identify posets with digraphs as
above without problems.

In analogy with the graph case in \cite{LSz,BCLSV1}, we define the
functional $t(Q,P)$ for finite posets as the proportion of all maps
$Q\to P$ that are poset homomorphisms.
We similarly also define $\tinj(Q,P)$ as
the proportion of all injective maps $Q\to P$ that are poset
homomorphisms 
and $\tind(Q,F)$ as
the proportion of all injective maps
$\gf:Q\to P$ such that $x<_Qy \iff \gf(x)<_P\gf(y)$
(\ie, $\gf$ is an isomorphism .onto an induced subposet of $P$).

We say that a sequence $(P_n)$ of
finite posets with $|P_n|\to\infty$ \emph{converges}, if $t(Q,P_n)$
converges for every finite poset $Q$.
(All unspecified limits in this paper are as  \ntoo.)
For completeness, we also say that a sequence $(P_n)$ of
finite posets with $|P_n|\not\to\infty$ {converges} if it is
eventually constant.

If a sequence of posets converge in this sense, what is its limit?
Exactly as for graph limits
\cite{LSz,BCLSV1,SJ209}, we may define limit objects in several
different, equivalent, ways. One possibility is to define the limit
objects as equivalence classes of convergent sequences, where two
convergent sequences $(P_n)$ and $(P_n')$ are defined to be equivalent
if the combined sequence $(P_1,P'_1,P_2,P'_2,\dots)$ converges. This
is similar to the standard construction of the completion of a metric
space using Cauchy sequences. In fact, it is easy to define a
metric on $\cP$ such that the Cauchy sequences are exactly the
convergent sequences, and then the poset limits are exactly the
elements of the completion. A simple way to construct such a metric is
to use one of the embedding in \refT{Ttau} of $\cP$ into a compact
metric space. Equivalently, and this is the method that we find
technically most convenient, we choose one of these embeddings, for
example $\ptaup:\cP\to\oipp=\oip\times\oi$ defined in \refS{Slim},
identify $\cP$ and its image $\ptaup(\cP)$, and let $\cpq$ be its
closure in $\oipp$; thus $\cpq$ is the set of \emph{poset limits}.
We also define
$\cpoo\=\cpq\setminus\cP$, the set of \emph{proper poset limits}.
Note that $\cpq$ is a compact metric space, because $\oipp$ is.
Further, $\cP$ is an open dense subset of $\cpq$, and thus $\cpoo$ is
a closed subset and thus itself a compact metric space.

It follows from this construction that
the functionals $t(Q,\cdot)$ , $\tinj(Q,\cdot)$ and $\tind(Q,\cdot)$ 
extends by continuity to $\cpq$ for every
$Q\in\cP$, and that $P_n\to\Pi\in\cpoo$ if and only if $|P_n|\to\infty$ and
$t(Q,P_n)\to t(Q,\Pi)$ for every $Q\in\cP$. As a consequence, a proper
poset limit $\Pi\in\cpoo$ is determined by $t(Q,\Pi)$, $Q\in\cP$.

Just as for graph limits, this construction is convenient for the
definition and existence of poset limits, but a more concrete
representation is desirable. We will study two such representations,
by \emph{kernels} and by \emph{exchangeable random posets}.

For graph limits, 
\citet{LSz} gave an important
(non-unique) representation by symmetric functions $W:\oi^2\to\oi$
(or, more generally, $W:\cS^2\to\oi$ for a probability space $\cS$),
see also \cite{BCLSV1,SJ209}.
(See \cite{SJ209} and \refS{SD} below for the more complicated version
for digraphs.) 
A similar construction for poset limits is as follows.

\begin{definition}\label{D1}
  An \emph{\ops} $\sfmuux$ is a probability space $\sfmu$ equipped with a
  partial order $\prec$ such that $\set{(x,y):x\prec y}$ is a measurable subset
  of $\cS\times\cS$ (\ie, belongs to the product $\gs$-field
  $\cF\times\cF$).

A \emph{kernel} on an \ops{} $\sfmuux$ is a measurable function
$W:\cS\times\cS\to\oi$ such that, for $x,y,z\in\cS$,
\begin{align}
  \label{w1}
W(x,y)>0&\implies x\prec y,\\
\label{w2}
W(x,y)>0 \text { and } W(y,z)>0 &\implies W(x,z)=1.
\end{align}

A \emph{strict kernel} is a kernel such that 
$W(x,y)>0\iff x\prec y$.
\end{definition}

When convenient, we may omit parts of the notation that are clear from
the context and say, \eg, that $\sss$ or
$(\sss,\mu)$ is a probability space or an \ops.

\begin{remark}\label{Rstrict}
We may when convenient suppose that the kernel is strict, by replacing
the order $\prec$ on $\cS$ by $\prec'$ defined by $x\prec'y$ if $W(x,y)>0$.
Note further that by \eqref{w2}, a strict kernel $W(x,y)$ is typically
determined to be 0 or 1 for many $(x,y)$; it is only when $(x,y)$
forms a gap in the order $\prec'$ that we have the freedom to choose
$W(x,y)\in(0,1)$. 
\end{remark}

Let $[n]\=\set{1,\dots,n}$ for  $n\in\bbN\=\set{1,2,\dots}$,
and $[\infty]\=\bbN$. Thus $[n]$ is a set of cardinality
$n$ for all $n\in\bbN\cup\set\infty$.

\begin{definition}\label{Dpnw}
Given a kernel $W$ on an \ops{} \sfmuux, we define for every
$n\in\bbnn$ a random poset $\pnw$ of cardinality $n$ by taking a
sequence $(X_i)_{i=1}^\infty$ of \iid{} points in $\cS$ with
distribution $\mu$, and independent uniformly distributed random
variables $\xiij\sim U(0,1)$, $i,j\in\bbN$, and then defining $\pnw$
to be $\xn$ with the partial order $\precy=\precpnw$ defined by: 
$i\precy j$ if and only if $\xiij<W(X_i,X_j)$.  
In other words, given $(X_i)$, we define the partial order randomly
such that $i\precy j$ with probability $W(X_i,X_j)$, with 
(conditionally) independent choices for different pairs $(i,j)$.
\end{definition}

Note that $\precy$ really is a partial order because of \eqref{w1},
which implies irreflexivity and asymmetry, and \eqref{w2}, which
implies transitivity. (This is the reason why we have to insist that
$W(x,z)=1$ in \eqref{w2}.) 
%In particular, note that \eqref{w1} implies $W(x,x)=0$ and 

\begin{remark}
  \label{Rpnw}
We insist in \refD{D1} that
\eqref{w1}--\eqref{w2} hold for \emph{all}
$x,y,z$, and not just a.e.; this will require some technical
arguments in proofs in \refS{Skernels} to replace a
candidate kernel by a kernel that is \aex{} equal to it. Note
that we can define $\pnw$ as above also if $W$ only satisfies
\eqref{w1}--\eqref{w2} \aex; $\pnw$ then
will be a poset a.s. (We will use this in the proof of
\refT{T1+} below.)
\end{remark}

\begin{example}\label{EW1}
  For any \ops, $W(x,y)=\ett{x\prec y}$ is a strict kernel. 
(We use $\ett{\cE}$ to denote the indicator function of the event
  $\cE$, which is 1 if $\cE$ occurs and 0 otherwise.)
In this case 
$i\precpnw j\iff  X_i\prec X_j$ 
and we  do not need the auxiliary random variables $\xiij$.
In other words, $\pnw$ then is (apart from the labelling) just the
subset \set{X_1,\dots,X_n} of $\cS$ with the induced order, provided
$X_1,\dots,X_n$ are distinct (or, in general, if we regard
\set{X_1,\dots,X_n} as a multiset). 

Note that every strict kernel with values in \setoi{} is of this
type. (In particular, every strict kernel on an \ops{} with a
continuous order.)
\end{example}

\begin{example}\label{Ep}
  Let $\sss=\setoi$ with $\mu\set0=\mu\set1=1/2$ and $0\prec 1$; let
  further $W(0,1)=p$ for some given $p\in\oi$, and, as required by
  \eqref{w1}, $W(0,0)=W(1,0)=W(1,1)=0$. Then $\pnw$ consists of a
  random 'lower' set of roughly half the vertices and a complementary
  'upper' set, and $u\precpnw v$ with probability $p$ for all lower $u$
  and upper $v$ (and never otherwise), independently for all pairs $(u,v)$.
\end{example}

Further examples are given below and in \refS{Sex}.

One of our main results is the following representation theorem,
parallel to the result for graph limits by \citet{LSz}. The proofs of
this and other theorems in the introduction are given in later sections.

\begin{theorem}
  \label{T1}
%  \begin{thmenumerate}
%\item
Every kernel $W$ on an \ops{} \sfmuux{} 
defines a poset limit \/ $\Pi_W\in\cpoo$ such that the following holds.
\begin{romenumerate}
  \item\label{T1a}
$\pnw\asto\Pi_W$ as \ntoo.
  \item\label{T1b}
$\displaystyle t(Q,\piw)=
\int_{\csq} \prod_{ij:i<_Q j} W(x_i,x_j) \dd \mu(x_1)\dots\dd\mu(x_\xQ),
\quad Q\in\cP 
$. \hfill 
$(\stepcounter{equation}\theequation
\makeatletter\xdef\@currentlabel{\theequation}\makeatother\label{t1b})
$
\end{romenumerate}  
%\item 
Moreover, every poset limit $\Pi\in\cpoo$ can be represented in this
way, \ie, $\Pi=\piw$ for some kernel $W$ on an \ops{} \sfmuux.	
 % \end{thmenumerate}
\end{theorem}

Unfortunately, the \ops{} and the kernel $W$ in \refT{T1} are not
unique (just as in the corresponding representation of graph
limits). We discuss the question of when two kernels represent the same
poset limit in \refS{Suniqueness}. We note, however, the following
important fact. 

\begin{theorem}\label{Tpoopi}
Let $\Pi\in\cpoo$ and $n\in\bbnn$. Then the random poset $\pnw$ has
the same distribution for every kernel $W$ on an \ops{} that
represents $\Pi$.
We may consequently define the random poset $\pnpi$ as $\pnw$ for any
kernel $W$ such that $\Pi_W=\Pi$.
\end{theorem}

It is easy to see that if $Q$ is a finite poset, $n\ge|Q|$, 
and $W$ is a kernel, then
\begin{equation}
  \label{tnw}
\E \tinj(Q,\pnw)
=\int_{\csq} \prod_{ij:i<_Q j} W(x_i,x_j) \dd \mu(x_1)\dots\dd\mu(x_\xQ),
\end{equation}
the integral in \eqref{t1b}. Hence,
for every poset limit $\Pi$, finite poset $Q$ and finite $n\ge|Q|$,
\begin{equation}
  \label{tqpin}
\E \tinj(Q,\pnp) = t(Q,\Pi).
\end{equation}
If $Q$ is a finite labelled poset with ground set $\subset\bbN$ we
similarly find
\begin{equation}
  \label{tqpi}
\P\bigpar{Q\subset\poopi} = t(Q,\Pi).
\end{equation}
This is easily seen to be equivalent to
(see \eqref{ttinj}--\eqref{tinjind} and \eqref{ml1}--\eqref{ml2})
\begin{equation}
  \label{tqpind}
\P\bigpar{\poopi\rest n=Q} 
=\P\bigpar{\pnp=Q} 
= \tind(Q,\Pi),
\end{equation}
for every (labelled) poset $Q$ on $\nn$,
which describes the distribution of $\poopi$.

We can use the non-uniqueness of the representation to our advantage
by imposing further conditions (normalizations) that may be useful in
various situations.

\begin{theorem}
  \label{T1+}
We may in \refT{T1} choose one of the following further conditions
and impose it on the representing kernel $W$:
\begin{romenumerate}
  \item\label{t1+s}
$W$ is a strict kernel.
\item\label{t1+oi}
$\sfmu=\oi$ with Lebesgue measure;
\ie, $W$ is a kernel on $\oibglx$, where
  $\cB$ is the Borel $\gs$-field, $\gl$ is
  Lebesgue measure and $\prec$ is some (measurable) partial
  order, not necessarily the standard order.
\item\label{t1+oi2}
$\sfmu=\oi^2$ with Lebesgue measure, and $(x_1,y_1)\prec(x_2,y_2)$ if and
only if $x_1<x_2$ in the standard order.
\end{romenumerate}
\end{theorem}

When $\mu$ is the Lebesgue measure $\gl$ (in one or
several dimensions), we take $\cF$ 
to be the Borel $\gs$-field.
(We could use the Lebesgue $\gs$-field instead; this would not make
any essential difference since a Lebesgue measurable function into
$\oi$ is \aex{} equal to a Borel measurable function.)

We have, however, not yet been able to see whether it always is
possible to use 
$\sfmu=\oi$ with Lebesgue measure and the standard order $<$.
(This would supersede both \ref{t1+oi} and
\ref{t1+oi2} in \refT{T1+}, and yield a simplified representation of
poset limits.) We state this as an open problem:

\begin{problem}\label{P1}
Can every proper poset limit be represented by a kernel on
$\oibgl$, with the standard order $<$?
%$\sfmu=\oi$ with Lebesgue measure and the standard order $<$?  
\end{problem}

\begin{example}[Continuation of \refE{Ep}]
  \label{Ep2}
Let $\sss$ and $W$ be as in \refE{Ep}, and let $Q=\sss=\set{0,1}$.
\refT{T1}\ref{T1b} then yields 
$$
t(Q,\piw)=\int_{\sss^2}W(x,y)\dd\mu(x)\dd\mu(y)=p/4.$$
This shows that different $p$ yield different $\piw$. Consequently,
$\cpoo$ is uncountable.
\end{example}

\begin{example}
  \label{EP}
Let $P$ be a finite poset. Take $\sss=P$ and let the probability
measure $\mu$ be the uniform distribution on $P$: $\mu\set i=|P|\qw$
for every $i\in P$. Then $P$ becomes an \ops, and 
$W_P(x,y)\=\ett{x<_Py}$ is a strict kernel on $P$, see \refE{EW1}.
For any $Q\in\cP$, \refT{T1}\ref{T1b} shows that
$
t(Q,\piwp)=\P(x_i<_p x_j \text{ when $i<_Qj$})$
for \iid{} random vertices $x_i$ in $\cP$, which is just the
probability that the random mapping $i\mapsto x_i$ is a poset
homomorphism. Thus, writing $\pip\=\piwp$,
\begin{equation}\label{tp}
  t(Q,\pip)=t(Q,P)
\end{equation}
for all $Q\in\cP$.

We have shown that for every finite poset $P$ there is a poset limit
$\pip\in\cpoo$ such that \eqref{tp} holds for all $Q\in\cP$. Note that
this defines $\pip$ uniquely; however, the map $P\mapsto\pip$ is not
injective, as is shown by the example $\setoi\times\nn$ discussed
further in \refS{Slim} or the trivial posets in \refE{E0}.
Note also that the mapping is not surjective, since $\cP$ is countable
and $\cpoo$ is uncountable (\eg, by \refE{Ep2}). Hence only some
(exceptionally simple) poset limits can be represented as $\pip$ for a
finite poset $P$.
\end{example}

We can now state a convergence criterion in terms of 
the cut metric defined in \refS{Scut}. (See \cite{BCLSV1} for
the graph version.)

\begin{theorem}
    \label{Tcutp}
Let $(P_n)$ be a sequence of finite posets with
$|P_n|\to\infty$
and let $\Pi\in\cpoo$. Let $W_{P_n}$ be the kernel
defined by $P_n$ as in \refE{EP}, and let $W$ be
any kernel that represents $\Pi$. 
Then, as \ntoo,
$P_n\to\Pi\iff\dcut(W_{P_n},W)\to0$.
\end{theorem}

Our second representation of graph limits uses \exch{} random posets.

\begin{definition}
  A random infinite poset (or digraph) on $\oo$ is \emph{\exch} if its
  distribution is invariant under every permutation of $\oo$.

Similarly,
an array $\set{I_{ij}}_{i,j=1}^\infty$, of random
variables is
\emph{(jointly) exchangeable} if 
the array $\set{I_{\gs(i)\gs(j)}}_{i,j=1}^\infty$ has the same
distribution as $\set{I_{ij}}_{i,j=1}^\infty$ for
every permutation $\sigma$ of $\oo$.
\end{definition}

Consequently, if $R$ is a random poset on $\oo$ and 
$I_{ij}\=\ett{i<_Rj}$, 
then $R$ is \exch{} if and only if the array \set{I_{ij}} is.

The random poset \poow{} defined in \refD{Dpnw} is evidently \exch,
and thus so is $\poopi$ in \refT{Tpoopi}. More generally, we can
construct \exch{} random infinite posets by taking mixtures of such
distributions, \ie, by taking $\poow$ or $\poopi$ with a random kernel
$W$ or a random graph limit $\Pi$ (which of course is assumed to be independent
of the other random variables $X_i$ and $\xiij$ in the construction);
\cf{} the classical de Finetti's theorem for \exch{} sequences of
random variables, see \eg{} \citet[Theorem 1.1]{Kallenberg:exch}.
Another of our main results is that this yields all \exch{} random
infinite posets, which can be seen as a de Finetti theorem for posets.
(It is a special case of the general representation theorem for
\exch{} arrays by Aldous and Hoover
\cite{Aldous,Hoover,Kallenberg:exch}.
Cf.\ the graph case in \cite{SJ209}.)
Moreover, the poset limits correspond to \exch{} random infinite
posets whose distribution is an extreme point in the set of all such
distributions, and this yields a unique representation of poset
limits as follows.
\begin{theorem}
  \label{TE}
  \begin{thmenumerate}
\item\label{terandom}
There is a one-to-one correspondence between distributions of random
elements $\Pi\in\cpoo$ and 
distributions of
\exch{} random infinite posets
$R\in\cpoo$ given by
$R\eqd\poopi$; this relation between $\Pi$ and $R$ is
equivalent to either of 
\begin{align}
  \label{te}  
\E t(Q,\Pi)&=\P(R\supset Q)\\
\intertext{or}
\E\tind(Q,\Pi)&=\P(R|_A=Q)   \label{tea}  
\end{align}
for every finite labelled poset $Q$ with a ground set\/ $A\subset\oo$. 
Furthermore, then $R\rest n\dto \Pi$ in $\cpq$ as \ntoo.
\item\label{tex}
There is a one-to-one correspondence between 
poset limits $\Pi\in\cpoo$ and extreme points of the set of
distributions of \exch{} random infinite posets
$R$.
This correspondence is given by 
$R\eqd\poopi$, or, equivalently,
either of
\begin{align}
 t(Q,\Pi)&=\P(R\supset Q)
\label{ce}
\intertext{or}
\tind(Q,\Pi)&=\P(R|_A=Q)   
\label{cea}  
\end{align}
for every finite labelled poset $Q$ with a ground set\/ $A\subset\oo$. 
Furthermore, then $R\rest n\asto \Pi$ in $\cpq$ as \ntoo.
  \end{thmenumerate}
\end{theorem}

We can characterize these extreme point distributions of \exch{}
random infinite posets as follows.
Let  $\cploo$ be the space of all (labelled) infinite posets on $\oo$.
This can be seen as a subset of the product space
$\set{0,1}^{\bbN\times\bbN}$, using indicators $I_{ij}$ as above. 
We
equip this product space with the product topology, which is compact
and metric; then  
$\cploo$ is a closed subset and thus itself a compact metric space.

\begin{theorem}
  \label{TE2}
Let $R$ be an \exch{} random infinite poset.
Then the following are equivalent.
\begin{romenumerate}
  \item\label{te2a}
The distribution of $R$ is an extreme point in the set of \exch{}
distributions in the space $\cploo$ of all labelled infinite
posets on $\oo$.
  \item\label{te2b'}
If $Q_1$ and $Q_2$ are two finite posets with disjoint ground sets
contained in $\bbN$, then
\begin{equation*}
  \P(R\supset Q_1\cup Q_2)=\P(R\supset Q_1)\P(R\supset Q_2).
\end{equation*}
(Here $Q_1\cup Q_2$ denotes the poset with ground set
$Q_1\gr\cup Q_2\gr$ and $x<_Qy\iff
x<_{Q_1}y\text{ or } x<_{Q_2}y$; in particular
$x\not<_Qy$ if $x\in Q_1\gr$ and $y\in
Q_2\gr$ or conversely.)
  \item\label{te2b}
The restrictions $R\rest k$ and $R\restx{[k+1,\infty)}$ are
  independent for every $k$.
  \item\label{te2c}
Let $\cF_n$ be the $\gs$-field generated by
$R\restx{[n,\infty)}$. Then the tail $\gs$-field
  $\bigcap_{n=1}^\infty \cF_n$ is trivial, \ie, contains only events with
  probability $0$ or $1$.
\end{romenumerate}
\end{theorem}

There is also a more direct relation between poset limits and \exch{}
random infinite posets, without going through kernels and $\poopi$.
Poset limits are limits of unlabelled finite posets. For labelled
finite posets there is another, more elementary, notion of a limit as
an infinite poset. More precisely, if $P_n$ is a labelled poset on the
ground set $[N_n]$ for some finite $N_n$ with $N_n\to\infty$ as \ntoo,
and $R$ is a poset on $\oo$, we say that $P_n\to R$ if, for every pair
$(i,j)\in\bbN^2$, $\ett{i<_{P_n} j}\to\ett{i<_R j}$, \ie, 
if $i<_Rj$ then $i<_{P_n}j$ for all large $n$ and 
if $i\not<_Rj$ then $i\not<_{P_n}j$ for all large $n$.
Equivalently, we may regard each $P_n$ as an element of the space
$\cploo$ of posets on $\bbN$ by adding an infinite number of points
incomparable to everything else (in fact, any extension to $\bbN$
would do, but it seems natural to choose the trivial one); then
$P_n\to R$ in this sense just means convergence in $\cploo$ with the
topology just introduced. For unlabelled posets, we can always choose
a labelling. Of course, the choice of labelling may affect the result,
so we choose a random labelling. Thus, if $P$ is a finite unlabelled
poset, we let $\hP$ be the labelled poset obtained by randomly
labelling $P$ by $1,\dots,|P|$, with the same probability $1/|P|!$ for
each possible labelling. As above, we can also extend $\hP$ to a
random poset on $\oo$, which we by abuse of notation still denote by
$\hP$.
The appropriate limit for a sequence $(P_n)$ of finite posets then is
limit in distribution of $(\hP_n)$ as random elements of the compact
metric space $\cploo$. This turns out to be equivalent to convergence
of the unlabelled posets $P_n$ as defined above (and in \refD{Dconv}
below), \ie, in $\cpq$.

\begin{theorem}\label{TC1}  Let $(P_n)$ be a sequence of finite
  unlabelled posets  and   assume that 
  $|P_n|\to\infty$. Then the following are equivalent.
  \begin{romenumerate}
\item
$P_n\to \Pi$ in $\cpq$ for some  $\Pi\in\cpq$.	
\item
$\hP_n\dto R$ in $\cploo$ for some random $R\in\cploo$.	
  \end{romenumerate}
If these hold, then
$R$ is \exch, and $R\eqd\poopi$; consequently,
\eqref{ce} and \eqref{cea} hold
%$\P\xpar{R \supset Q} =  t(Q,\Pi)$ 
for every finite labelled poset $Q$ with a ground set\/ $A\subset\oo$. 
\end{theorem}

Finally, we note that by regarding posets as digraphs, we obtain an embedding
$\cP\subset\cD$ which extends to an embedding $\cpq\subset\cdq$. The
poset limits can thus be seen as special digraph limits. We
characterize the digraph limits that are poset limits in several ways
in \refT{TPQ}.

Sections \ref{Sprel}--\ref{Slim} contain definitions
and some basic properties of poset limits. 
The theorems above are proven in Sections
\ref{Sexch}--\ref{Skernels} and \ref{Scut2}.
The cut metric is defined and studied in Sections
\refand{Scut}{Scut2};
in particular we show that it makes the set of all kernels into a
compact metric space, which is homeomorphic to $\cpoo$
(Theorems \refand{Tcut2}{Tcutw}.)
The (lack of) uniqueness of the representation by kernels is discussed
in \refS{Suniqueness}, where conditions for equivalence are given.
Further examples are given in \refS{Sex}. 
The relation between poset limits and digraph limits is discussed
further in \refS{SD}, and the final \refS{Sfurther}
contains further comments.

\section{Preliminaries}\label{Sprel}

We consider both labelled and unlabelled posets and digraphs. 
We use for convenience $[n]$ as our standard
ground set for labelled posets and vertex set for labelled digraphs,
\ie, we use the labels $1,2,\dots$.

A \emph{digraph} ({directed graph}) $G$ consists of a vertex set
$V(G)$ and an edge set $E(G)\subseteq V(G)\times V(G)$; 
the edge indicators thus form an arbitrary zero--one matrix
\set{X_{ij}}, $i,j\in V(G)$. 
We let $|G|$ denote the number of vertices.
Unless we state otherwise explicitly, we assume that $1\le
|G|<\infty$, but we will also sometimes consider infinite digraphs. 

Let, for $n\in\bbN$,  
$\DLn$ be the set of the $2^{n^2}$ labelled digraphs
with vertex set $[n]$ and let $\DUn$ be the 
set of unlabelled digraphs with $n$
vertices; $\DUn$ can formally be defined as the quotient set
$\DLn/\cong$ modulo isomorphisms.
Further, let
$\DL\=\bigcup_{n\ge1}\DLn$ and $\DU\=\bigcup_{n\ge1}\DUn$; thus $\DU$
is the set of finite unlabelled graphs.

Similarly,
let %, for $n\in\bbN$,  
$\PLn$ be the set of all posets
with ground set $[n]$ and let $\PUn$ be the quotient
set $\PLn/\cong$ of unlabelled posets with $n$ vertices, and let
$\PL\=\bigcup_{n\ge1}\PLn$ and $\PU\=\bigcup_{n\ge1}\PUn$, the set of
finite unlabelled posets.

As said above, we can regard every poset as a digraph. This works for
both labelled and unlabelled posets and yields the
inclusions 
$\PLn\subset\DLn$,
$\PUn\subset\DUn$,
$\PL\subset\DL$,
$\PU\subset\DU$.
Further, every labelled poset or digraph can be regarded as an
unlabelled one by ignoring the labels. Hence it often does not matter
whether the posets and digraphs are labelled or not, but we shall be
explicit the times it does matter.

We can characterize the digraphs that are posets  using a
few special digraphs.
Let, for $n\ge1$,  $\sC_n$ be the directed cycle with $n$ vertices and
$n$ edges,
and let $\sP_n$ be the directed path with $n+1$ vertices and
$n$ edges. 
(Thus $\sC_1$ is a loop and $\sC_2$ a double edge.) 
We regard these as unlabelled digraphs.  
%Let $\sC_n$, $n\ge1$, be the directed cycle with $n$ vertices $[n]$ and
%$n$ edges $i\toi+1$ modulo $n$, 
%and let $\sP_n$, $n\ge1$, be the directed path with $n+1$ vertices $[n+1]$ and
%$n$ edges $i\toi+1$, $1\le i\le n$. 
Note that these, except $\sP_1$, are \emph{not} posets.
Moreover, if $G$ is a digraph, consider the relation $i\to j$, meaning
that there is an edge from $i$ to $j$. This relation is irreflexive if
and only if $G$ contains no loop, \ie{} no subgraph $\sC_1$. Similarly,
it is asymmetric if and only if $G$ contains no double edge, \ie{} no
$\sC_2$. Assuming these properties of $G$, 
if $x,y,z$ are three vertices such that $x\to y$ and $y\to z$, then
necessarily $x$, $y$ and $z$ are distinct, and either $z\to x$ or
$\set{x,y,z}$ induces a subgraph $\sP_2$ or $\sC_3$; consequently, 
the relation then is transitive
if and only if there is no such induced subgraph. We have proven the following
characterization.

\begin{lemma}\label{L1}
A (finite or infinite) digraph $G$ is a poset if and only if it does
not have any induced subgraph %of the type
$\sC_1$, $\sC_2$, $\sC_3$, or $\sP_2$. \nopf
\end{lemma}

\section{Digraph and poset limits}\label{Slim}

We repeat some of the notation and results for digraphs in
\cite{SJ209} and give corresponding results for posets.

If $G$ is an (unlabelled) digraph and $v_1,\dots,v_k$ is a sequence of
vertices in $G$, then $G(v_1,\dots,v_k)$ denotes the labelled digraph
with vertex set $[k]$ where we put an edge $i\to j$ if
$v_i\to v_j$ in $G$. 
We allow the possibility that $v_i=v_j$ for some $i$ and $j$. 

We let $G\xk$, for $k\ge1$, be the random digraph $G(v_1,\dots,v_k)$
obtained by sampling $\vvk$ uniformly at random among the vertices of
$G$, with replacement. In other words, $\vvk$ are independent
uniformly distributed random vertices of $G$.

For $k\le |G|$, we further let $G\xk'$ be the random digraph $G(\vvki)$
where we sample $\vvki$ uniformly at random without replacement; the
sequence $\vvki$ is thus a uniformly distributed random sequence of
$k$ distinct vertices.
Hence, $G\xk'$ is the induced subgraph on a random set of $k$
vertices, with the vertices relabelled $1,\dots,k$.

For a finite poset $P$, we similarly define
$P(v_1,\dots,v_k)$, $P\xk$ and $P\xk'$ (the latter if $k\le|P|$);
these are posets with ground set $\xk$, and 
$P\xk$ and $P\xk'$ are random. Note that these definitions are
consistent with our identification of cosets and (certain) digraphs: 
for example, $P\xk$ is the same for the poset $P$
as for $P$ regarded as a digraph.

The graph limit theory in \cite{LSz} and subsequent papers is based on 
the study of the functional $t(F,G)$ 
which is defined for two graphs $F$ and $G$ as the proportion of all
mappings $V(F)\to V(G)$ that are
graph homomorphisms $F\to G$.
In probabilistic terms, $t(F,G)$ is the probability that a uniform random 
mapping $V(F)\to V(G)$ is a graph homomorphism.
For the digraph version, see \cite{SJ209},
$\gf:V(F)\to V(G)$ is a homomorphism if $i\to j$ in $F$
implies $\gf(i)\to\gf(j)$ in $G$. Thus, using the notation just
introduced and
assuming that $F$ is labelled and $k=|F|$,
we can write the definition as
\begin{equation}
  \label{t}
t(F,G)\=\P\bigpar{F\subseteq G\xk}.
\end{equation}
Note that both $F$ and $G\xk$ are digraphs on the same vertex set $\xk$,
so the relation $F\subseteq G\xk$ is well-defined as $E(F)\subseteq E(G\xk)$.
We further define, again 
following \cite{LSz} (and the notation of \cite{BCLSV1} and \cite{SJ209}),
%but stating the definitions in different but equivalent forms,
 with $k=|F|$ as in \eqref{t}, 
\begin{align}
  \tinj(F,G)&\=\P\bigpar{F\subseteq G\xk'},\label{tinj}
\intertext{the proportion of injective maps 
$V(F)\to V(G)$ that are graph homomorphisms, % $F\to G$.
and}
  \tind(F,G)&\=\P\bigpar{F= G\xk'},  \label{tind}
\end{align}
provided $F$ and $G$ are digraphs with $|F|\le |G|$.
If $|F|>|G|$ we set $\tinj(F,G)\=\tind(F,G)\=0$.
Note that although the relations $F\subseteq G\xk$, $F\subseteq G\xk'$
and $F= G\xk'$ 
may depend on the
labelling of $F$, the probabilities  in \eqref{t}--\eqref{tind} do
not, by symmetry, 
so $t(F,G)$, $\tinj(F,G)$   and $\tind(F,G)$  are
well defined  for unlabelled $F$ and
$G$ (by choosing any labellings).

The definitions \eqref{t}--\eqref{tind} can be used for finite posets
too. Thus, if $P$ and $Q$ are finite (unlabelled) posets, then $t(P,Q)$,
$\tinj(P,Q)$ and $\tind(P,Q)$ are defined as numbers in $\oi$. Note
that these numbers are the same as if we regard $P$ and $Q$ as
digraphs; we will therefore use the same notation for the poset case
as for the digraph case.

The basic definition of \Lovasz{} and Szegedy 
\cite{LSz} and 
\Borgsetal{} \cite{BCLSV1} 
is that a sequence $(G_n)$ of graphs
converges if $t(F,G_n)$ converges for every graph $F$. 
As in \cite{SJ209}, we modify this by requiring also that $|G_n|$
converges to some finite or infinite limit.
We let,
as in \cite{SJ209}, 
$\cdd$ be the union of $\cD$ and some one-point set \set{*} and
define
the mappings
$\tau,\tauinj,\tauind:\DU\to\oid$ 
and
$\taup:\cD\to\oidd=\oid\times\oi$ by
\begin{align}%\label{tau}
  \tau(G)&\=(t(F,G))_{F\in\cD}\in\oid,
\\
  \tauinj(G)&\=(\tinj(F,G))_{F\in\cD}\in\oid,
\\
  \tauind(G)&\=(\tind(F,G))_{F\in\cD}\in\oid,
\\
  \taup(G)&\=\bigpar{\tau(G),\,|G|\qw}\in\oidd.
\end{align}
For posets we similarly define
$\cpp\=\cP\cup\set{*}$ and the mappings
$\ptau,\ptauinj,\ptauind:\PU\to\oip$ and 
$\ptaup:\cP\to\oipp=\oip\times\oi$
%$\ptauinj:\PU\to\oip$ and
%$\ptauind :\PU\to\oip$
by considering $F$ in $\cP$ only; these mappings can thus be obtained
from $\tau,\tauinj,\tauind,\taup$ by a projection selecting some
coordinates only. 

The mappings $\tau$ and $\ptau$ are not injective on $\cP$. For example (the
poset version of an example in \cite{LSz} and \cite{BCLSV1}), the posets
$\set{0,1}\times \xn$ with $(x_1,y_1)<(x_2,y_2)$ if $x_1<x_2$ have the
same images under $\tau$ and $\ptau$ for all $n\in\bbN$. However, it
is easy to see
that $\taup,\tauinj$ and $\tauind$ are injective on $\DU$,
\cf{} \cite{SJ209}, and, similarly, that
$\ptaup,\ptauinj$ and $\ptauind$ are injective on $\PU$, 
(This uses the special definitions of $\tauinj(F,G)$ and
$\tauind(F,G)$ 
when $|F|>|G|$.)

Although the mappings $\ptaup, \ptauinj, \ptauind$ contain only part of the
information in $\taup$ and so on, the injectivity of them shows that
they in fact contain all possible information. This is also seen in
the following stronger result concerning limits.

\begin{theorem}\label{Ttau}
  Suppose that $P_n$ is a sequence of finite posets.
Then the following conditions are equivalent.
  \begin{romenumerate}
\item
$\ptaup(P_n)$ converges in $\oipp$, \ie{}	
$t(Q,P_n)$ converges for every poset $Q\in\cP$
and $|P_n|$ converges to some limit in $\bbNx$.
\item
$\ptauinj(P_n)$ converges in $\oip$, \ie{}	
$\tinj(Q,P_n)$ converges for every poset $Q\in\cP$.
\item
$\ptauind(P_n)$ converges in $\oip$, \ie{}	
$\tind(Q,P_n)$ converges for every poset $Q\in\cP$.
\item
$\taup(P_n)$ converges in $\oidd$, \ie{}	
$t(F,P_n)$ converges for every digraph $F\in\cD$
and $|P_n|$ converges to some limit in $\bbNx$.
\item
$\tauinj(P_n)$ converges in $\oid$, \ie{}	
$\tinj(F,P_n)$ converges for every digraph $F\in\cD$.
\item
$\tauind(P_n)$ converges in $\oid$, \ie{}	
$\tind(F,P_n)$ converges for every digraph $F\in\cD$.
  \end{romenumerate}
\end{theorem}

\begin{proof}
  It is easily seen that each of the conditions implies that $|P_n|$
  converges to a limit in $\bbnn$.
Further, if $|P_n|$ converges to a finite limit, each of the
  statements implies that $P_n=P$ for all sufficiently large $n$ and
  some (unlabelled) poset $P\in\cP$.

It thus suffices to consider the case $|P_n|\to\infty$.
In this case, for every $F\in\cD$,
$t(F,P_n)-\tinj(F,P_n)=O(|F|^2/|P_n|)\to0$, 
see \cite{LSz} and \cite{SJ209}, 
and thus (i)$\iff$(ii) and (iv)$\iff$(v).

Further, see \cite{BCLSV1}, \cite{LSz} or \cite{SJ209} for the easy details,
one can go between the two families
$\set{\tinj(F,\cdot)}_{F\in\DU}$ and 
$\set{\tind(F,\cdot)}_{F\in\DU}$ of  functionals on $\DU$
by summation and inclusion-exclusion,
and for posets a similar argument holds for the families
$\set{\tinj(F,\cdot)}_{F\in\PU}$ and 
$\set{\tind(F,\cdot)}_{F\in\PU}$;
hence it follows that (ii)$\iff$(iii) and (v)$\iff$(vi).

Finally, (iii)$\iff$(vi) because $\tind(F,P_n)=0$ for every digraph
$F$ that is not a poset.
\end{proof}

\begin{definition}
  \label{Dconv}
A sequence $(P_n)$ of finite posets \emph{converges} if
one, and thus all, of the conditions in \refT{Ttau} holds.
\end{definition}

\begin{remark}
  \label{Rsala}
As seen in the proof of \refT{Ttau}, the case when
$|P_n|\not\to\infty$ is not very interesting since then $(P_n)$
converges if and only if the sequence is eventually constant. The
interesting case is thus $|P_n|\to\infty$, and then convergence of
$(P_n)$ is also equivalent to convergence of $\ptau(P_n)$ in $\oip$ or
$\tau(P_n)$ in $\oid$.
\end{remark}

Since $\ptaup$ is injective, 
we can identify $\cP$ with its image $\ptaup(\cP)\subseteq\oipp$
and define $\cpq\subseteq\oipp$ as its closure.
Alternatively, we can consider $\ptauinj$ or $\ptauind$;
we can again identify $\cP$ with its image
and consider its closure $\cpq$ in $\oip$.
It follows from \refT{Ttau} that 
the three compactifications $\overline{\ptaup(\cP)}$, 
$\overline{\ptauinj(\cP)}$, $\overline{\ptauind(\cP)}$ are homeomorphic
and we can use any of them for $\cpq$.
Moreover, we can also, again by \refT{Ttau}, use $\ptau$, $\ptauinj$
or $\ptauind$ and embed $\cP$ in $\oidd$ or $\oid$ and obtain $\cpq$
as a compact subset of $\oidd$ or $\oid$. This is equivalent to
regarding posets as digraphs and using the embeddings
$\cP\subset\cD\subset\cdq$ and defining $\cpq$ as the closure of $\cP$
in $\cdq$. (Thus $\cpq$ can be regarded as a subset of $\cdq$.)
Since all these constructions yield homeomorphic results it does not
matter which one we use.
Note that $\cpq$ is a compact metric space. Different, equivalent,
metrics are given by the embeddings above into $\oipp$, $\oip$,
$\oidd$, $\oid$.

We let $\cpoo\=\cpq\setminus\cP$; 
this is the set of all limit objects of sequences
$(P_n)$ in $\cP$ with $|P_n|\to\infty$; \ie, $\cpoo$ is the set of all
proper poset limits.

For every fixed digraph $F$, the functions $t(F,\cdot)$,
$\tinj(F,\cdot)$ and $\tind(F,\cdot)$ have unique continuous
extensions to $\cdq$, for which we use the same notation. 
In particular, $t(Q,\Pi)$ is defined for every finite poset $Q$ and
poset limit $\Pi\in\cpq$.
We similarly extend $|\cdot|\qw$ continuously
to $\cdq$ by defining $|\gG|=\infty$
and thus $|\gG|\qw=0$ for $\gG\in\cdoo\=\cdq\setminus\cD$.
It is easily seen that 
\begin{equation}
  \label{ttinj}
\tinj(F,\gG)=t(F,\gG)
\end{equation}
for every $F\in\cD$
and $\gG\in\cdoo$ \cite{SJ209}; in particular for $F\in\cP$ and
$\gG\in\cpoo$.
Moreover, for any $Q,P\in\cP$,
$\tinj(Q,P)=\sum_{Q'\supseteq Q}\tind(Q',P)$,
where we sum over all posets $Q'\supseteq Q$ with the same
ground set $Q\gr$, and thus by continuity
\begin{equation}
\label{tinjind}
\tinj(Q,\Pi)=\sum_{Q'\supseteq Q}\tind(Q',\Pi)
\end{equation}
for every $Q\in P$ and $\Pi\in\cpq$.

Thus $\cpoo=\set{\Pi\in\cpq:|\Pi|\qw=0}$, which shows that $\cpoo$ is
a closed and thus compact subset of $\cpq$. Conversely, $\cP$ is an open
subset of $\cpq$; by \refR{Rsala}, it has the discrete topology.
Note further that $\cP$ is countable 
while $\cpq$ and $\cpoo$ are uncountable,
\eg{} by \refE{Ep2}.

We summarize the results above on convergence.

\begin{theorem}\label{T1old}
  A sequence $(P_n)$ of finite posets converges in the sense of \refD{Dconv}
if and only if it converges in the compact metric space
  $\cpq$.
\nopf
\end{theorem}

The construction of $\cP$ further immediately implies the
following related characterization of convergence in $\cpoo$.

\begin{theorem}
  \label{T1ox}
A sequence $\Pi_n$ of proper graph limits (\ie, elements
of $\cpoo$)
converges [to a proper graph limit $\Pi$]
if and only if $t(Q,\Pi_n)$ converges [to $t(Q,\Pi)$]
for every finite poset $Q$.

We can here replace $t$ by $\tinj$ or $\tind$;
further, we may let $Q$ range over all finite digraphs instead
of posets. \nopf
\end{theorem}

\section{Exchangeable random infinite posets}\label{Sexch}

It is straightforward to verify that Sections 3--5 of \citet{SJ209}
hold with only notational changes for the poset case as well as for
the graph case treated there. Rather than repeating the details, we
therefore omit them and refer to \cite{SJ209}, 
giving only a few comments.
We first
obtain the following basic result on convergence in distribution of
\emph{random} unlabelled posets, corresponding to \cite[Theorem 3.1]{SJ209}.

\begin{theorem}
  \label{T2}
Let $P_n$, $n\ge1$, be random unlabelled posets and assume that
$|P_n|\pto\infty$.
The following are equivalent, as \ntoo.
\begin{romenumerate}
  \item\label{T2a}
$P_n\dto \Pi$ for some random $\Pi\in\cpq$.
  \item\label{T2b}
For every finite family\/ $Q_1,\dots,Q_m$ of (non-random) finite posets, the
random variables $t(Q_1,P_n),\dots,t(Q_m,P_n)$ converge jointly in
distribution. 
  \item\label{T2c}
For every (non-random) $Q\in\cP$, the
random variables $t(Q,P_n)$ converge in
distribution.
  \item\label{T2d}
For every (non-random) $Q\in\cP$, the
expectations $\E t(Q,P_n)$ converge.
\end{romenumerate}
If these properties hold, then the limits in \ref{T2b},
\ref{T2c} and \ref{T2d} are $\bigpar{t(Q_i,\Pi)}_{i=1}^m$, $t(Q,\Pi)$ and
$\E t(Q,\Pi)$, respectively; conversely, if 
 \ref{T2b}, \ref{T2c} or \ref{T2d} holds with these limits for some
 random $\Pi\in\cpq$, then \ref{T2a} holds with the same $\Pi$.
Furthermore, $\Pi\in\cpoo$ a.s.

The same results hold if $t$ is replaced by $\tinj$ or $\tind$.
\nopf
\end{theorem}

Using this we then obtain \refT{TC1}, which corresponds to
\cite[Theorems 4.1 and 5.2]{SJ209};
note that \eqref{cea} is the poset version of a formula in
\cite[Theorem 4.1]{SJ209}, which follows because, if $n$ is so large
that $A\subseteq\nn$, $\P(\hpn|_A=Q)=\tind(Q,P_n)\to\tind(Q,\Pi)$, 
and that
\eqref{ce} easily follows from \eqref{cea}  by summing over
$Q'\supseteq Q$ on the same ground set.
We really cannot prove 
the equality $R\eqd\poopi$ 
yet, since we have defined $\poopi$ using kernels and \refT{T1}, which
is not yet proven. Instead, we note only that $R\eqd\poopi$ will
follow by \eqref{ce} and \eqref{tqpi} 
or \eqref{cea} and \eqref{tqpind} once we have proven \refT{T1} and thus
verified \eqref{tqpi} and \eqref{tqpind} in \refS{Skernels}.
(Alternatively, we could have used \eqref{tqpi} and \eqref{tqpind} as
a definition of $\poopi$.) 

\begin{remark}
Actually,
\cite[Theorems 4.1]{SJ209} is stated more generally for sequences of
random graphs, and similarly
\refT{TC1} extends to the case of random finite posets $P_n$ with
$|P_n|\pto\infty$; then the limit $\Pi\in\cpq$ is in general random
too, and 
(i) becomes $P_n\dto\Pi$ while \eqref{ce} and \eqref{cea} have to be
replaced by \eqref{te} and \eqref{tea}.
\end{remark}

We then obtain \refT{TE}, which corresponds to 
\cite[Theorems 5.3 and Corollary 5.4]{SJ209}, and 
\refT{TE2}, which corresponds to 
\cite[Theorems 5.5]{SJ209}.
The \as{} convergence of $R\rest n=\poopi\rest n$ in
\refT{TE}\ref{tex} follows, as in \cite[Remark 5.1]{SJ209}, 
because $\tinj(Q,R\rest n)$,
$n\ge|Q|$, is a reverse martingale for every $Q\in\cP$.

\section{Kernels}\label{Skernels}

\begin{proof}[Proof of \refT{T1}]

First, let $W$ be a kernel on an \ops. Then $R=\poow$ defined in
\refD{Dpnw} is an \erip, which satisfies the independence condition
\refT{TE2}\ref{te2b'}; hence, by \refT{TE2}\ref{te2a} 
its distribution is an extreme point in the set of
\exch{} distributions, and by \refT{TE}\ref{tex}
there exists a poset limit $\Pi$ (which we
call $\piw$) such that \eqref{ce} and \eqref{cea} hold, and
$\pnw=R\rest n\asto\Pi=\piw$ in $\cpq$. This proves \ref{T1a}.
Further, it follows directly from the definition of $\poow$ that if
$Q$ is a finite labelled poset, then
\begin{equation*}
\P\bigpar{\poow\supset Q} =
\int_{\csq} \prod_{ij:i<_Q j} W(x_i,x_j) \dd \mu(x_1)\dots \dd\mu(x_\xQ),
\end{equation*}
and thus \ref{T1b} follows by \eqref{ce}.

For the converse, suppose that $\Pi\in\cpoo$, and consider the
corresponding \erip{} $R$ given by \refT{TE}. (I.e., $\poopi$,
although we have not yet shown this, so we have to use
only \eqref{ce} and \eqref{cea} until \refT{T1} is proven.)
Let $\iij\=\ett{i<_R j}$, $i,j\in\oo$. Then $(\iij)$ is a 
jointly \exch{} random arrays of zero--one variables, with the
diagonal entries $I_{ii}=0$. For such \exch{} random arrays,
the Aldous--Hoover representation theorem takes the form, see 
\citet[Theorem 7.22]{Kallenberg:exch},
  \begin{align}\label{f4}
I_{ij}&=f(\xio,\xii,\xij,\xiij), \qquad i\neq j,	
  \end{align}
where $f:\oi^4\to\setoi$ is a (Borel) measurable
function, $\xi_{ji}=\xiij$, and 
$\xio$, $\xii$ ($1\le i$) and $\xiij$ ($1\le i<j$) are independent
random variables uniformly distributed on $\oi$. 
By \refT{TE}\ref{tex}, the distribution of the array $({I_{ij}})$ is an extreme
point in the set of \exch{} distributions, and thus  by \refT{TE2} and
\cite[Lemma 7.35]{Kallenberg:exch}, there exists such a representation
where $f$ does not depend on $\xio$, so \eqref{f4} becomes
$I_{ij}=f(\xii,\xij,\xiij)$, $i\neq j$.	
We then further define 
\begin{equation}\label{w0}
  W_0(x,y)\=\P\bigpar{f(x,y,\xi)=1}
=\E f(x,y,\xi),
\end{equation}
where $\xi\sim\uoi$. 
(In general, the variable $\xio$ can be interpreted as making $W$
random; this is needed if we consider a random $\Pi$ as in
\refT{TE}\ref{terandom}, but not in the present case.)

As our \ops{} we take \oilm, with an order to be defined later.
The function $W_0$ is almost the sought kernel, but not quite. The
problem is that the function $f$, and thus $W_0$, can be arbitrarily
changed on a null set without affecting the distribution of $(\iij)$;
consequently we can only show properties such as \eqref{w2} a.e. for $W_0$.
We thus have to make a suitable choice of $W$ among all functions that
are \aex{} equal to $W_0$. 

Recall that a point $(x,y)$ is a \emph{\lp} of an integrable function
$F$ on $\bbR^2$
if
\begin{equation}
  \label{leb}
(2\eps)^{-2}\hskip-1em\iint\limits_{|x'-x|<\eps,\,|y'-y|<\eps}
|F(x',y')-F(x,y)|\dd x'\dd y' \to0
\qquad \text{as }\eps\to0,
\end{equation}
and that \aex{} point is a \lp{} of $F$, see
\eg{} \citet[\S1.8]{Stein}. 
This applies trivially to functions defined on $\oiqq$ too, by
extending the functions to $\bbR^2$ by defining
them as $0$ outside $\oiqq$.
We modify the function $W_0$ in two steps. We first define
\begin{equation}  \label{limw1}
W_1(x,y)\=\liminf_{\eps\to0} 
(2\eps)^{-2}\hskip-1em\iint\limits_{|x'-x|<\eps,\,|y'-y|<\eps}
W_0(x',y')\dd x'\dd y',
\end{equation}
and note that $W_1=W_0$ at every \lp{} of $W_0$ and thus a.e.
Next, we let $E$ be the set of all Lebesgue points of $W_1$ in $\oiq^2$
and define $W(x,y)\=W_1(x,y)\ett{(x,y)\in E}$. Then 
$0\le W(x,y)\le 1$ and
$W=W_1=W_0$ a.e.
Moreover, if $W(x,y)>0$, then $W_1(x,y)=W(x,y)$,
$(x,y)\in\oiqq$ and $(x,y)$ is a \lp{}
of $W_1$; hence, using $W_1(x,y)=W(x,y)$ and $W_1=W$ \aex,
$(x,y)$ is a \lp{} of $W$. Finally, if $(x,y)\in\oiqq$ and 
\begin{equation}
  \label{lim1}
(2\eps)^{-2}\hskip-1em\iint\limits_{|x'-x|<\eps,\,|y'-y|<\eps}
W(x',y')\dd x'\dd y'
\to1
\end{equation}
as $\eps\to0$, then $W_1(x,y)=1$ by \eqref{limw1}; thus, 
using $W_1\le 1$, 
\eqref{lim1}  implies that
$(x,y)$ is a Lebesgue point of $W_1$, and hence $(x,y)\in E$ and
$W(x,y)=W_1(x,y)=1$.  

After these preliminaries, note that $I_{12}=I_{23}=1$ implies
$I_{13}=1$ since $R$ is a poset. Hence, using \eqref{f4} and
\eqref{w0},
and the independence of \set{\xi_i,\xi_{jk}},
\begin{equation*}
  \begin{split}
0&=
  \P(I_{12}=I_{23}=1,\,I_{13}=0)
\\&
=
\E \bigpar{ 
f(\xi_1,\xi_2,\xi_{12})f(\xi_2,\xi_3,\xi_{23})(1-f(\xi_1,\xi_3,\xi_{13}))}
\\
&=
\E \bigpar{W_0(\xi_1,\xi_2)W_0(\xi_2,\xi_3)(1-W_0(\xi_1,\xi_3))}	
\\
&=
\E \bigpar{W(\xi_1,\xi_2)W(\xi_2,\xi_3)(1-W(\xi_1,\xi_3))};
  \end{split}
\end{equation*}
thus 
\begin{equation}\label{www}
W(x_1,x_2)W(x_2,x_3)(1-W(x_1,x_3))=0 \quad \text{a.e.}
\end{equation}

Similarly, since $R$ does not contain a directed cycle $1<_R2<_R3<_R1$,
$P(I_{12}=I_{23}=I_{31}=1)=0$ and 
\begin{equation}\label{wc3}
W(x_1,x_2)W(x_2,x_3)W(x_3,x_1)=0
\quad \text{a.e.}
\end{equation}

Now assume that $x$, $y$ and $z$ are such that
$W(x,y)>0$ and $W(y,z)>0$. Let $\eps>0$ and let
$\xex$ be a random, uniformly
distributed, point in $(x-\eps,x+\eps)$ and let similarly $\xey$ and
$\xez$ be random points in $(y-\eps,y+\eps)$ and $(z-\eps,z+\eps)$;
these three variables being independent.
Since $W(x,y)>0$, $(x,y)$ is a \lp{} of $W$, and thus \eqref{leb} shows
that $\E|W(\xex,\xey)-W(x,y)|\to0$ as $\eps\to0$. In particular, using
Markov's inequality, $\P(W(\xex,\xey)=0)\to0$ as $\eps\to0$.
Similarly, $\P(W(\xey,\xez)=0)\to0$ as $\eps\to0$.
On the other hand, \eqref{www} implies
$W(\xex,\xey)W(\xey,\xez)(1-W(\xex,\xez))=0$ a.s., and thus
\begin{equation*}
  \P\bigpar{W(\xex,\xez)<1}
\le
\P(W(\xex,\xey)=0) +\P(W(\xey,\xez)=0)\to0,
\end{equation*}
as $\eps\to0$. It follows that \eqref{lim1} holds at $(x,z)$, and
thus, by the remarks above, $W(x,z)=1$. Consequently, 
\begin{align}
\label{w2r}
W(x,y)>0 \text { and } W(y,z)>0 &\implies W(x,z)=1,
\end{align}
which is \eqref{w2}.

Similarly, still assuming $W(x,y)>0$ and $W(y,z)>0$,
\eqref{wc3} implies
$W(\xex,\xey)W(\xey,\xez)W(\xez,\xex)=0$ a.s., and thus
\begin{equation}\label{ww}
  \P\bigpar{W(\xez,\xex)>0}
\le
\P(W(\xex,\xey)=0) +\P(W(\xey,\xez)=0)\to0,
\end{equation}
as $\eps\to0$. If further $W(z,x)>0$, then $(z,x)$ is a \lp{}
of $W$ and $\P(W(\xez,\xex)=0)\to0$ as $\eps\to0$, which contradicts
\eqref{ww}. Consequently, 
%$W(x,y)>0$ and $W(y,z)>0$ imply $W(z,x)=0$.
\begin{align}
\label{w3}
W(x,y)>0 \text { and } W(y,z)>0 &\implies W(z,x)=0.
\end{align}

Now suppose that $W(x,x)>0$ for some $x$. Taking $y=z=x$ in  
\eqref{w3} we find $W(x,x)=0$, a contradiction.
Hence, $W(x,x)=0$ for every $x$.
Further, if both $W(x,y)>0$ and $W(y,x)>0$ for some $x$ and $y$, then
\eqref{w2r} yields $W(x,x)=1$, which was just shown to be impossible.
Hence $W(x,y)W(y,x)=0$ for all $x$ and $y$. 
These properties and \eqref{w2r} show that we may define a partial
order $\prec$ on $\sss=\oi$  by $x\prec y$ if $W(x,y)>0$, and then 
$W$ is a (strict) kernel on the \ops{} $\oibglx$, where $\gl$ is the
\lm. 
%($\cF$ may be chosen as either the Borel or Lebesgue $\gs$-field; 
(We took $f$ Borel measurable, and then $W_0$, $W_1$, $E$ and
$W$ are Borel measurable too.)

Finally, it follows from \eqref{f4} and \eqref{w0} that for every
finite poset $Q$ with ground set $A\subset\bbN$,
\begin{equation}\label{wp}
  \begin{split}
  \P\xpar{Q\subset R}
&=
\P\biggpar{ \prod_{ij:i<_Q j} \iij=1}
=
\E \prod_{ij:i<_Q j} \iij
=
\E \prod_{ij:i<_Q j} f(\xii,\xij,\xiij)
\\&
=
\E \prod_{ij:i<_Q j} W_0(\xii,\xij)
=
\E \prod_{ij:i<_Q j} W(\xii,\xij)
\\&=
\int_{\csq} \prod_{ij:i<_Q j} W(x_i,x_j) \dd \mu(x_1)\dots
\dd\mu(x_\xQ).
  \end{split}
\raisetag{1.5\baselineskip}
\end{equation}
Hence, by \eqref{ce} and \eqref{t1b},
$t(Q,\Pi)=\P(Q\subset R) = t(Q,\piw)$ for all such posets
$Q$, and thus $\piw=\Pi$.
\end{proof}
\begin{remark}
Remember that we have $\xiij=\xiji$, which in principle may give a
dependence between $ij$ and $ji$ terms. This is an important
complication in other situations, for example for digraphs
\cite{SJ209}, but is of no concern for posets, where at most one of
$W(\xii,\xij)$ and $W(\xij,\xii)$ is non-zero, and similarly, in
\eqref{wp}, at most one of $i<_Qj$ and $j<_Qi$ holds.  
\end{remark}

As remarked above, it now follows that $R\eqd\poopi$ in Theorems
\refand{TE}{TC1}, for example by 
\eqref{tqpind} and \eqref{cea}, or directly by \eqref{wp}.

\begin{remark}
  Alternatively, 
we can regard $R$ as an exchangeable random infinite digraph,
and use the representation by a quintuple of functions
$\bW=(W_{00},W_{01},W_{10},W_{11},\wo)$ 
as in \citet[Theorem 9.1]{SJ209}, see \refS{SD}; here
$\wab:\oi^ 2\to\oi$ and
$w:\oi\to\oi$. The function $w$ generates loops
and $W_{11}$ generates doubly directed edges (\ie,
cycles $\sC_2$); hence $w=0$ and $W_{11}=0$
in the poset case. Further, $W_{01}(x,y)=W_{10}(y,x)$ and
$\sum_{\ga,\gb=0}^1\wab(x,y)=1$, so the quintuple $\bW$ is
determined by $W_{10}$. We then can replace
\eqref{w0} by $W_0\=W_{10}$, and complete the
proof by adjusting $W_0$ on a null set as above.
\end{remark}

\begin{lemma}
  \label{Lg}
Let $\sfmuux$ be an \ops, and let
$g(x)\=\mu\set{z\in\sss:z\prec
x}$. Then, the set 
\set{(x,y)\in\sss^ 2:x\prec y\text{ and } g(x)\ge g(y)} is a null set in
$\sss^ 2$.
\end{lemma}

\begin{proof}
  Let, for $x\in\sss$, 
$D_x\=\set{z:z\prec x}$ %(thus $g(x)=\mu(D_x)$) and 
and $E_x\=\set{z:z\prec x \text{ and }g(z)\ge g(x)}$.
%=\set{z:z\prec x \text{ and }g(z)= g(x)}$ (since 
%$z\prec x \implies D_z\subseteq D_x \implies g(z)\le g(x)$).
%
If $z\in E_x$, then $z\prec x$ and thus $D_z\subseteq D_x$, and
$\mu(D_z)=g(z)\ge g(x)=\mu(D_x)$; hence $g(z)=g(x)$ and
$\mu(D_x\setminus D_z)=0$. In particular, then
$\mu(E_x\setminus D_z)=0$, because $E_x\subseteq D_x$.

For two points $y,z\in
E_x$, at least one of $y\notin D_z$  and $z\notin D_y$ holds, and thus
by symmetry
\begin{equation*}
 \mu(E_x)^ 2 
\le 2 \int_{ E_x}\int_{E_x} \ett{y\notin
  D_z}\dd\mu(y)\dd\mu(z)
= 2\int_{E_x} \mu(E_x\setminus D_z)\dd\mu(z)=0.
\end{equation*}
Hence, $\mu(E_x)=0$ for every $x$, and thus
\begin{equation*}
\mu\times \mu\set{(z,y)\in\sss^ 2:z\prec y\text{ and } g(z)\ge g(y)} 
=\int_{\sss} \mu(E_y)\dd\mu(y)=0.
\qedhere
\end{equation*}
\end{proof}

\begin{proof}[Proof of \refT{T1+}]
  The proof of \refT{T1} above gives a kernel satisfying \ref{t1+s} and
  \ref{t1+oi}.

For \ref{t1+oi2}, we start with a kernel $W_1$ on an
\ops{} $\oibglx$ as in \ref{t1+oi}; thus $\prec$ is some
partial order on $\oi$, in general different from the
standard order $<$.
Define $g(x)\=\gl\set{z\in\oi:z\prec x}$. 
Then $x\preceq y\implies g(x)\le g(y)$.
Moreover, by \refL{Lg}, for \aex{} $(x,y)$,
$x\prec y\implies g(x)<g(y)$.

Let $W_2(x,y)\=W_1(x,y)\ett{g(x)<g(y)}$; this too is a kernel on $\oibglx$. 
Since
$W_1(x,y)>0\implies x\prec y\implies g(x)<g(y)$ for \aex{} $(x,y)$
 by \refL{Lg}, we
have  $W_2=W_1$ a.e., and thus $\Pi_{W_2}=\Pi_{W_1}=\Pi$.
Moreover, 
\begin{equation}
  \label{ww2}
W_2(x,y)>0\implies g(x)<g(y).
\end{equation}

Let $U_1,U_2\sim \uoi$ be independent uniform random
variables.
Let $\xi\=g(U_1)$, and let 
$h:\oi\to\oi$ be the right-continuous inverse of its distribution function
$s\mapsto \P(g(U_1)\le s)$;
then $h$ is a non-decreasing function such that $h(U_1)\eqd
g(U_1)=\xi$.

By the transfer theorem \cite[Theorem 6.10]{Kallenberg} with
$\xi\=g(U_1)$, $\eta\=U_1$,
$\txi\=h(U_1)\eqd\xi$, there exists a measurable
function $f:\oi^ 2\to\oi$ such that if
$\teta\=f(\txi,U_2)$, then 
$(\txi,\teta)\eqd(\xi,\eta)=(g(U_1),U_1)$.
This implies $\txi-g(\teta)\eqd\xi-g(\eta)=0$
and thus
$\txi=g(\teta)$ a.s., \ie{}
\begin{equation*}
  h(U_1)=g(\teta)=g(f(\txi,U_2))=g(f(h(U_1),U_2))
\qquad \text{a.s.};
\end{equation*}
hence,
\begin{equation}\label{hgf}
  h(x_1)=g(f(h(x_1),x_2))
\qquad\text{\aex{} on $\oi^ 2$.}
\end{equation}

Let $\sss\=\oi^ 2$ with Lebesgue measure, and define the functions
$W_3,W_4:\sss^ 2\to\oi$ by
\begin{align}
  \label{ww3}
W_3\bigpar{(x_1,x_2),(y_1,y_2)}
&\=
W_2\bigpar{f(h(x_1),x_2),f(h(y_1),y_2)}
\intertext{and}
W_4\bigpar{(x_1,x_2),(y_1,y_2)}
&\=
W_3\bigpar{(x_1,x_2),(y_1,y_2)}\ett{x_1<y_1}.
\end{align}
Then $W_4$ is a kernel on $\sfmuux$.
Further, if $W_3\bigpar{(x_1,x_2),(y_1,y_2)}>0$, then
\eqref{ww3} and \eqref{ww2} yield
$g(f(h(x_1),x_2))<g(f(h(y_1),y_2))$, which by
\eqref{hgf} implies, except on a null set in
$\sss^2=\oi^4$, that
$h(x_1)<h(y_1)$ and thus, since $h$ is non-decreasing,
$x_1<y_1$. Consequently, $W_3=W_4$ \aex{} on
$\sss^2=\oi^4$.

Since $f(h(U_1),U_2)=f(\txi,U_2)=\teta\eqd\eta=U_1$
is uniformly distributed on \oi, it follows from the construction
of $\pnw$ that $P(n,W_4)\eqd P(n,W_3)\eqd P(n,W_2)$ for every
$n\le\infty$. Thus by \refT{T1}(i) or
\ref{TE}\ref{tex},
$\Pi_{W_4}=\Pi_{W_2}=\Pi$
and $W_4$ is a kernel on $\sfmuux$ that
represents $\Pi$.
\end{proof}

\begin{proof}[Proof of \refT{Tpoopi}]
  Let $W$ be a kernel with $\piw=\Pi$. By
  \eqref{tnw}, \eqref{t1b} and \eqref{ttinj}, for every finite
  poset $Q$ and $n\ge|Q|$,
  \begin{equation}
\label{ml1}
\E\tinj(Q,\pnw))=t(Q,\piw)=t(Q,\Pi)=\tinj(Q,\Pi).	
  \end{equation}
By \eqref{tinjind}, it follows that for any finite $n$
and labelled poset $Q$ on $\nn$,
  \begin{equation}
\label{ml2}
\P(Q=\pnw)=
\E\tind(Q,\pnw))=\tind(Q,\Pi).	
  \end{equation}
Hence, the distribution of $\pnw$ is determined by
$\Pi$ for finite $n$, and does not depend on the choice
of $W$. Further, the distribution of $\poow$ is
determined by the distribution of $\pnw=\poow\rest n$,
$1\le n<\infty$, so this distribution too is determined by $\Pi$.
\end{proof}

Moreover, \eqref{tqpin}--\eqref{tqpind} follow from \eqref{ml1}--\eqref{ml2}.

\section{Cut norm and metric}\label{Scut}

In this section  it will be convenient to (usually) ignore orders and
study general \ps{s}.

Let $(\sss,\mu)$ be a probability space.
%(We will always assume that $\sss$ is a \ps{} in this section.)
We define
the \emph{cut norm} $\cn{W}$ of $W\in L^1(\sss^2)$ 
by, see \cite{FKquick,BCLSV1,SJ223},
\begin{equation}\label{cnone}
 \cnone W \= \sup_{S,T}
  \Bigabs{\int_{S \times T} W(x,y) \dd\mu(x)\dd\mu(y)},
\end{equation}
where the supremum is taken over all pairs of measurable subsets of $\sss$.
Alternatively, one can take
\begin{equation}\label{cntwo}
 \cntwo W \= \sup_{\sn{f},\sn{g}\le1}
  \Bigabs{\int_{\sss^2} f(x)W(x,y)g(y)\dd\mu(x)\dd\mu(y)}.
\end{equation}
It is easily seen that
$ \cnone{W} \le \cntwo{W} \le 4\cnone{W}$;
thus the two norms $\cnone{\cdot}$ and $\cntwo{\cdot}$ are equivalent.
It will for our purposes not be important which one we use, and
we shall write $\cn{\cdot}$ for either norm.
(There are further, equivalent versions of the cut norm; see
\cite{BCLSV1}.)
Note that for either definition of the cut norm we have
$ \bigabs{\int W} \le \cn{W} \le \on{W}$.

If $W$ is a function defined on $\sss^2$ for some space
$\sss$, and $\gf:\sss'\to\sss$ is a
function, we define the function $W^\gf$ on
$\sss'^2$ by
\begin{equation}
  \label{wgf}
W^\gf(x,y)=W\bigpar{\gf(x),\gf(y)}.
\end{equation}
Given two integrable functions $W_j:\sss_j^2\to\bbR$,
$j=1,2$, where $\sss_1$ and $\sss_2$ are two,
in general different, probability spaces, we define the \emph{cut metric}
\cite{BCLSV1} by 
\begin{equation}\label{dcut}
  \dcut(W_1,W_2)=\inf_{\gf_1,\gf_2}\cn{W_1^{\gf_1}-W_2^{\gf_2}},
\end{equation}
taking the infimum over all pairs $(\gf_1,\gf_2)$ of
measure preserving maps $\gf_1:\sss\to\sss_1$ and
$\gf_2:\sss\to\sss_2$ defined on a common
probability space $(\sss,\mu)$. (See further \cite{BCLSV1} and
\cite{BR} where some equivalent versions are given and discussed, and
\refL{Lcut4} below.)
%When necessary, we write
%$\dcutone$ and $\dcuttwo$ to distinguish between the
%two different versions given by taking \eqref{cnone} or
%\eqref{cntwo} in \eqref{dcut}.
Note that
\begin{equation}
  \label{dcut0}
\dcut(W,W^\gf)=0 
\end{equation}
for every $W$ and every measure
preserving $\gf$; this is the point of using $\dcut$.
Clearly, $0\le\dcut(W_1,W_2)<\infty$ and
$\dcut(W_1,W_2)=\dcut(W_2,W_1)$.
The triangle inequality holds too, so $\dcut$ is a
semimetric (but not a metric because of \eqref{dcut0}); this
is not quite obvious so for completeness we give a proof (which is
longer than we would like), first giving
another simple lemma.

We say that a function $W$ on $\sss^2$ (where $\sss$ is a  probability space)
is of \emph{finite type} if there exists a finite measurable
partition $\sss=\cap_{i=1}^ N A_i$ such that $W$ is
constant on each $A_i\times A_j$. 

\begin{lemma}
  \label{Lcut1}
If $\sss$ is a \ps{} and $W\in
L^1(\sssq)$, then for every $\eps>0$ there exists a
finite type $W'\in L^1(\sssq)$ such that
\begin{equation*}
  \dcut(W,W') \le \cn{W-W'}\le \on{W-W'}<\eps.
\end{equation*}
\end{lemma}

\begin{proof}
The set of finite type functions is dense in $L^1(\sss^2)$ 
by standard integration theory, so we may choose $W'$ with
$\on{W-W'}<\eps$. The first two inequalities are
immediate from the definitions.  
\end{proof}

\begin{lemma}\label{Lcut3}
  For any probability spaces $\sss_\ell$ and integrable
  functions $W_\ell:\sss_\ell\times\sss_\ell\to\bbR$,
  $\ell=1,2,3$, we have the triangle inequality:
  \begin{equation*}
\dcut(W_1,W_3)\le \dcut(W_1,W_2)+\dcut(W_2,W_3).	
  \end{equation*}
\end{lemma}

\begin{proof}
It is easy to see
  that 
$\dcut(U_1,U_2)\le\dcut(V_1,V_2)+\on{U_1-V_1}+\on{U_2-V_2}$
for any integrable functions $U_1,V_1,U_2,V_2$ defined on the
corresponding spaces. Hence, 
if $W'_\ell:\sss_\ell^2\to\bbR$ are finite type
  functions,
  \begin{multline*}
\dcut(W_1,W_3)-\dcut(W_1,W_2)-\dcut(W_2,W_3)
\\
\le\dcut(W'_1,W'_3)-\dcut(W'_1,W'_2)-\dcut(W'_2,W'_3)
+2\sum_{\ell=1}^3\norm{W_\ell-W_\ell'}_{L^1(\sssq_\ell)},	  
	  \end{multline*}
so by \refL{Lcut1}, 
it suffices to prove the triangle
  inequality for finite type functions $W_\ell$.

Thus, assume now that $W_1,W_2,W_3$ are finite type functions, with
corresponding partitions
$\set{A_i}_{i=1}^{N_1}$ ,
$\set{B_i}_{i=1}^{N_2}$,
$\set{C_i}_{i=1}^{N_3}$ of
$\sss_1,\sss_2,\sss_3$, respectively.
Suppose further that $\eps>0$ and that
$\gf_1:\sss\to\sss_1$ and
$\gf_2:\sss\to\sss_2$ are measure preserving with
$\cn{W_1^{\gf_1}-W_2^{\gf_2}}\le\dcut(W_1,W_2)+\eps$,
and similarly that
$\gf'_2:\sss'\to\sss_2$ and
$\gf'_3:\sss'\to\sss_3$ are measure preserving with
$\cn{W_2^{\gf'_2}-W_3^{\gf'_3}}\le\dcut(W_2,W_3)+\eps$.
Our task is to couple the two couplings $(\gf_1,\gf_2)$
and $(\gf'_2,\gf'_3)$, which seems difficult in general,
but is simple in the finite type case.

It is easy to see that if $W$ is of finite type and constant on
the sets $A_i\times A_j$ for a partition \set{A_i},
then the integrals in \eqref{cnone} and
\eqref{cntwo} are maximized by considering $S$ and $T$ that are
unions of some sets $A_i$, and $f$ and $g$ that are
constant on each $A_i$. Consequently,  $\cn{W}$
depends only on the values $W(A_i\times A_j)$ and the
measures $\mu(A_i)$.
Since $W_1^{\gf_1}-W_2^{\gf_2}$ is finite type
with partition
$\set{\gf_1\qw(A_i)\cap\gf_2\qw(B_j)}_{i,j}$
of $\sss$, it follows that
$\cn{W_1^{\gf_1}-W_2^{\gf_2}}$
depends only on $W_1$, $W_2$ and the measures 
$\mu\bigpar{\gf_1\qw(A_i)\cap\gf_2\qw(B_j)}$.
The corresponding holds for $\cn{W_2^{\gf'_2}-W_3^{\gf'_3}}$.

Define (with $0/0$ interpreted as 0)
\begin{equation*}
 a_{ij}\=
\frac{\mu\bigpar{\gf_1\qw(A_i)\cap\gf_2\qw(B_j)}}{\mu_1(A_i)\mu_2(B_j)} 
\end{equation*}
and
\begin{equation*}
 a'_{jk}\=
\frac{\mu'\bigpar{{\gf'_2}\qw(B_j)\cap{\gf'_3}\qw(C_k)}}{\mu_2(B_j)\mu_3(C_k)},
\end{equation*}
and note that, provided $\mu_2(B_j)\neq0$,
\begin{equation}\label{bx}
  \sum_i a_{ij}\mu_1(A_i)=\frac{\mu\bigpar{\gf_2\qw(B_j)}}{\mu_2(B_j)}=1
\end{equation}
and, similarly,
\begin{equation}\label{by}
  \sum_k a'_{jk}\mu_3(C_k)=\frac{\mu\bigpar{{\gf'_2}\qw(B_j)}}{\mu_2(B_j)}=1.
\end{equation}
Define a measure $\nu$ on
$\sssx\=\sss_1\times\sss_2\times\sss_3$ by
\begin{equation*}
  \nu(E)=\sum_{i,j,k}
  a_{ij}a'_{jk}\mu_1\times\mu_2\times\mu_3\bigpar{E\cap(A_i\times
  B_j\times C_k)},
\end{equation*}
and let $\pi_\ell:\sssx\to\sss_\ell$ be the
projection. Then, by \eqref{by},
\begin{equation*}
  \begin{split}
\nu\bigpar{\pi_1\qw(A_i)&\cap\pi_2\qw(B_j)}
=\nu(A_i\times B_j\times \sss_3)
\\&
=\sum_ka_{ij}a'_{jk}\mu_1(A_i)\mu_2(B_j)\mu_3(C_k)
=a_{ij}\mu_1(A_i)\mu_2(B_j)
\\&
=\mu\bigpar{\gf_1\qw(A_i)\cap\gf_2\qw(B_j)}.
  \end{split}
\end{equation*}
Hence, by the comments above,
$\cn{W_1^{\gf_1}-W_2^{\gf_2}}=\cn{W_1^{\pi_1}-W_2^{\pi_2}}$.
Similarly,
$\cn{W_2^{\gf'_2}-W_3^{\gf'_3}}=\cn{W_2^{\pi_2}-W_3^{\pi_3}}$.
Consequently,
\begin{equation*}
  \begin{split}
\dcut(W_1,W_3)
&\le \cn{W_1^{\pi_1}-W_3^{\pi_3}}
\le
\cn{W_1^{\pi_1}-W_2^{\pi_2}} +
	\cn{W_2^{\pi_2}-W_3^{\pi_3}}
\\&
\le\dcut(W_1,W_2)+\dcut(W_2,W_3)+2\eps,
  \end{split}
\end{equation*}
which completes the proof since $\eps>0$ is arbitrary.
\end{proof}

We let our kernels, and in this section more general functions, be
defined on arbitrary \ps{s}. Sometimes it is convenient to
use the special space $\oib$. (For simplicity we write often
$\oi$ instead of $\oib$. Thus, $\oi$ is
assumed to be equipped with Lebesgue measure unless we state otherwise.)
The next lemma shows that this can be done without loss of generality.
%(At least when using the cut metric.)

\begin{lemma}
  \label{Lcut0}
If $W\in L^1(\sssq)$ for some \ps{}
$\sss$, then there exists a function $W'\in
L^1(\oii)$ with $\dcut(W,W')=0$.
\end{lemma}

\begin{proof}
  It is shown in \citet[Proof of Theorem 7]{SJ210} first that
  there exists a function
  $h:\sss\to D\=\setoi^\infty$ and a function
  $V:D^2\to\bbR$ such that $W=V^h$, and secondly
  that if $\nu$ is the measure on $D$ that makes
  $h$ measure preserving, then there exists a measure preserving
  $\gf:\oi\to(D,\nu)$. Take
  $W'\=V^\gf$. Then
  $\dcut(W,V)=\dcut(V^h,V)=0$ and
  $\dcut(V,W')=\dcut(V,V^\gf)=0$, so
  $\dcut(W,W')=0$ by \refL{Lcut3}.
\end{proof}

\begin{lemma}
  \label{Lcut4}
If $W_1,W_2\in L^1(\oii)$, then 
\begin{equation*}
  \dcut(W_1,W_2)=\inf_\gf\cn{W_1-W_2^{\gf}},
\end{equation*}
taking the infimum over all \mpx{} bimeasurable bijections $\gf:\oi\to\oi$.
\end{lemma}

\begin{proof}
  By definition,
$\dcut(W_1,W_2)\le\cn{W_1-W_2^{\gf}}$
  for every such $\gf$.

Conversely, by \refL{Lcut1} again, it suffices to consider
finite type $W_1$ and $W_2$.
Thus, suppose that $W_1$ and $W_2$ are finite type with
corresponding partitions $\set{A_i}$ and \set{B_j}
of $\oi$. If
$\gf_1,\gf_2:\sss\to\oi$ are measure
preserving, where $(\sss,\mu)$ is any \ps, then, as
remarked in the proof of \refL{Lcut3},
$\cn{W_1^{\gf_1}-W_2^{\gf_2}}$
depends only on the numbers
$b_{ij}\=\mu\bigpar{\gf_1\qw(A_i)\cap\gf_2\qw(B_j)}$. 
Since
$\sum_jb_{ij}=\mu\bigpar{\gf_1\qw(A_i)}=\gl(A_i)$
and
$\sum_ib_{ij}=\mu\bigpar{\gf_2\qw(B_j)}=\gl(B_j)$,
we may partition each $A_i$ as $\bigcup_jA_{ij}$ and
each $B_j$ as $\bigcup_iB_{ij}$ with
$\gl(A_{ij})=\gl(B_{ij})=b_{ij}$. We may
then construct $\gf$ such that $\gf$ is a measure
preserving bijection of $A_{ij}$ onto $B_{ij}$
for all $i,j$ (possibly excepting some null sets; these are
easily handled). Then
$\gf\qw(B_j)=\gf\qw\bigpar{\bigcup_i B_{ij}}=\bigcup_i A_{ij}$,
and thus $A_i\cap\gf\qw(B_j)=A_{ij}$ and
\begin{equation*}
  \gl\bigpar{A_i\cap\gf\qw(B_j)}=\gl(A_{ij})=b_{ij}
=\mu\bigpar{\gf_1\qw(A_i)\cap\gf_2\qw(B_j)}. 
\end{equation*}
Consequently, with $\iota$ the identity function,
$\cn{W_1-W_2^{\gf}}$=
$\cn{W_1^{\iota}-W_2^{\gf}}$=
$\cn{W_1^{\gf_1}-W_2^{\gf_2}}$, and the result follows.
\end{proof}

By \refL{Lcut3}, the relation $W\cong W'$ if
$\dcut(W,W')=0$ defines an equivalence relation between
functions $W$, possibly defined for different \ps{s}.
We let, for a \ps{} $\sss$, $\cW(\sss)$
be the set af all measurable $W:\sssq\to\oi$, and let
$\bcw$ be the quotient space of
$\bigcup_\sss\cW(\sss)$ modulo $\cong$.
(The careful reader might correctly object that the collection of all
\ps{s} is not a set, so $\bigcup_\sss$ is not
defined. However, \refL{Lcut0} implies that it actually
suffices to consider $\cwoi$ modulo $\cong$, or the
union for
$\sss$ in any set of \ps{s} containing $\oi$.)

It follows from \refL{Lcut3} that $\bcwcut$ is a
metric space. The following important result is a minor variation of
the symmetric version in \citet{LSz:Sz}; for completeness we
give a proof although it is essentially the same as in the symmetric case.

\begin{theorem}
  \label{Tcut1}
$\bcwcut$ is a compact metric space.
\end{theorem}

\begin{proof}
  Recall that a metric space is compact if and only if it is complete
  and totally bounded.

We first show that $\bcwcut$ is complete. It suffices to show
that if $(W_n)_{n=1}^\infty$ is a sequence in
$\cwoi$ such that
$\dcut(W_n,W_{n+1})<2^{-n}$, then
$\dcut(W_n,W)\to0$ for some $W\in\cwoi$.

We choose, using \refL{Lcut4}, \mpx{} mappings
$\gf_n:\oi\to\oi$ such that
$\cn{W_n-W_{n+1}^{\gf_n}}<2^{-n}$.
Define inductively $\psi_1\=\iota$ (the identity on
$\oi$) and
$\psi_{n+1}\=\gf_n\circ\psi_n$; then
$W_{n+1}^{\psi_{n+1}}=\bigpar{W_{n+1}^{\gf_n}}^{\psi_{n}}$
and
\begin{equation*}
\cn{W_n^{\psi_n}-W_{n+1}^{\psi_{n+1}}}
=\cn{(W_n-W_{n+1}^{\gf_{n}})^{\psi_n}}
=\cn{W_n-W_{n+1}^{\gf_{n}}}
<2^{-n}.
\end{equation*}
Hence the functions $W_n'\=W_n^{\psi_n}$ form a
Cauchy sequence in $L^1(\oii,\cn{\ })$. Moreover,
each $W_n'$ is in the unit ball of
$L^\infty(\oii)=L^1(\oii)^*$, so by sequential
weak-$*$ compactness (which holds because $L^1(\oii)$
is separable),
there exists $W\in L^\infty(\oii)$ such that
$W_n\towx W$.

The assumption $0\le W_n'\le1$ implies
$0\le\iint W_n'(x,y)\etta_A(x)\etta_B(y)\le\gl(A)\gl(B)$ 
for all measurable $A$ and $B$. Since $W_n'\towx W$, this implies 
$0\le\iint
W(x,y)\etta_A(x)\etta_B(y)\le\gl(A)\gl(B)$ for all
$A$ and $B$, and thus by Lebesgue's differentiation
theorem (see \eg{} \citet[\S1.8]{Stein} again), $0\le W\le1$ a.e.;
hence we may assume $W\in\cwoi$. 

For all $f,g\in L^{\infty}(\oi)$ with $\sn
f, \sn g \le1$, the function $f(x)g(y)$ belongs to
$L^1(\oii)$, and thus the weak-$*$ convergence implies
(using for definiteness $\cntwo\cdot$)
\begin{equation*}
  \begin{split}
\Bigabs{\iint&\bigpar{W_n'(x,y)-W(x,y)}f(x)g(y)\dd x\dd y}	
\\&
=
\Bigabs{\lim_{m\to\infty}\iint\bigpar{W_n'(x,y)-W'_m(x,y)}f(x)g(y)\dd x\dd y}
\\&
\le\limsup_{m\to\infty}\cn{W_n'-W_m'}
\le 2^{1-n}.
  \end{split}
\end{equation*}
Taking the supremum over $f$ and $g$ we find
$\cn{W_n'-W}\le2^{1-n}$, and thus
\begin{equation*}
  \dcut(W_n,W)
=  \dcut(W_n',W)
\le
\cn{W_n'-W}\to0.
\end{equation*}
This proves the completeness of $\bcw$.

We next show that $\bcwcut$ is totally bounded. 
Let, for $N\ge1$, $K_N$ be the set of finite type
functions in $\cwoi$ with a partition with at most $N$
parts; we regard $K_N$ as a subset of $\bcw$. 

Let $\eps>0$. As in the proof in \citet[Section 4]{LSz:Sz} 
of Lemma 3.1 there (but now taking
  $\cK_n\=\set{\etta_{S\times T}}$ in
  Lemma 4.1 there
  rather than   $\cK_n\=\set{\etta_{S\times
  S}}$ as in the symmetric case given there), each
  $f\in\cwoi$ has distance at most $\eps$ to
  $K_N$ with $N\=\floor{\eps\qww}$.

By an obvious rearrangement, each element of $K_N$ has a
representation with a partition of $\oi$ into $N$
intervals. Let $A:=\set{(s_1,\dots,s_{N-1}):0\le
s_1\le\dots\le s_{N-1}\le1}$ and
$B\=\oi^{N^2}$. Thus, the function $f:A\times
B\to K_N$ given by
\begin{equation*}
  f\bigpar{s_1,\dots,s_{N-1},(a_{ij})_{i,j=1}^N}
\=\sum_{i,j=1}^N a_{ij}\etta_{(s_{i-1},s_i)}(x)\etta_{(s_{j-1},s_j)}(y),
\end{equation*}
with $s_0\=0$ and $s_N\=1$, is thus onto
$K_N$; further, $f$ is continuous into
$L^1(\oii)$ and thus into $\bcwcut$. Consequently,
$K_N=f(A\times B)$ is a continuous image of a compact set,
and thus $K_N$ is a compact subset of $\bcw$.
Hence, there exists a finite subset $F$ of $K_N$ such that
every point in $K_N$ has distance at most $\eps$ to
$F$. Consequently, every point in $\bcw$ has distance
at most $2\eps$ to $F$. Since $F$ is arbitrary,
this shows that $\bcw$ is totally bounded.
\end{proof}

We use the construction in \refD{Dpnw} for an arbitrary
$W\in\cwsss$; in general, $\precy$ will not be a
partial order so $\pnw$ will not be a poset, but we can
always regard $\pnw$ as a random digraph (with $i\precy
j$ interpreted as a directed edge $ij$).
We further define 
\begin{equation}  \label{tfw1}
t(F,W)\=
\int_{\csf} \prod_{ij\in F} W(x_i,x_j) \dd \mu(x_1)\dots\dd\mu(x_\xF)
\end{equation}
for every $W\in\cwsss$ and every finite digraph $F$;
thus \eqref{t1b} says that $t(Q,\piw)=t(Q,W)$ for
every finite poset $Q$ and kernel $W$.
Equivalently,
\begin{equation}
  \label{tfw}
t(F,W)=\P(i\precy j\text{ for every edge $ij$ in $F$}),
\end{equation}
where $\precy$ is the relation in $\poow$.

We say that a digraph is \emph{simple} if it can be obtained
by orienting a simple graph; in other words, a digraph is simple if it
has no loops or double edges (\ie, no induced $\sC_1$ or
$\sC_2$).
In particular, a poset is a simple digraph.

\begin{lemma}
  \label{Lcut2x}
Let $W_1\in\cw(\sss_1)$ and $W_2\in\cw(\sss_2)$ where
$\sss_1$ and $\sss_2$ are \ps{s}. Then, for every simple finite
digraph $F$,
if $m$ is the number of edges in $F$, then
\begin{equation*}
  \bigabs{t(F,W_1)-t(F,W_2)} \le m \dcut(W_1,W_2).
\end{equation*}
In particular, for every finite poset $Q$ (with $m$ the number
of pairs $(i,j)$ with $i<_Q j$),
\begin{equation*}
  \bigabs{t(Q,\piwi)-t(Q,\piwii)} \le m \dcut(W_1,W_2).
\end{equation*}
\end{lemma}

\begin{proof}
  This is identical to the proof in the symmetric case (when $F$
  is a finite undirected graph) given in \cite{BCLSV1} (with
  an unimportant extra factor in the constant);
see also \cite[Lemma 2.2]{BR}
for a nice formulation (with the constant given above).
\end{proof}

Note that we exclude digraphs $F$ with a loop or a double edge (an induced
$\sC_1$ or $\sC_2$) since we do not want factors of the type
$W(x_i,x_i)$ or
$W(x_i,x_j)W(x_j,x_i)$ in the integrals.
(In fact, \refL{Lcut2x} fails for $F=\sC_1$ or $\sC_2$.)

We now focus on functions $W\in\cwsss$ that are kernels
(recall \refD{D1}.
We define three special digraphs $\sD_1,\sD_2,\sD_3$
with vertex sets \set{1,2,3} and edge sets
$E(\sD_1)=\set{12,23}$,
$E(\sD_2)=\set{12,23,13}$ and $E(\sD_3)=\set{12,23,31}$.
(Thus $\sD_2$ is a poset, but not $\sD_1$ and $\sD_3$, and $\sD_3=\sC_3$.)

\begin{lemma}
  \label{Lcut5}
Let\/ $W\in\cwoi$.
Then the following are equivalent.
\begin{romenumerate}
  \item\label{tcut5a}
For every finite $n$, $\pnw$ is \as{} a poset.
  \item\label{tcut5b}
$\poow$ is \as{} a poset.
  \item\label{tcut5c}
There exists a partial order $\prec$ on $\oi$ and a
kernel $W'$ on $\oibglx$ such that $W=W'$ \aex
  \item\label{tcut5d}
$t(\sD_1,W)=t(\sD_2,W)$ and $t(\sD_3,W)=0$.
\end{romenumerate}
\end{lemma}

\begin{proof}
\ref{tcut5a}$\iff$\ref{tcut5b}.
is  obvious because $\pnw=\poow\rest n$.

\ref{tcut5c}$\implies$\ref{tcut5a},\ref{tcut5b}.
is clear since $\pnw=P(n,W')$ a.s.

\ref{tcut5b}$\implies$\ref{tcut5c}.
If \ref{tcut5b} holds, then $R\=\poow$ is an
\exch{} random infinite poset. We follow the proof of
\refT{T1} in \refS{Skernels}, noting that by
\refD{Dpnw},
$\iij\=\ett{\xiij<W(X_i,X_j)}$ so we already
have the representation \eqref{f4} (with
$\xii=X_i$), and \eqref{w0} yields
$W_0(x,y)\=\P\bigpar{\xi<W(x,y)}=W(x,y)$. The
remainder of the proof of \refT{T1} shows that we may modify
$W_0$ on a null set such that the result (denoted $W$
there and $W'$ here) is a kernel on $\oibglx$ for some
partial order $\prec$ on $\oi$.

\ref{tcut5d}$\implies$\ref{tcut5c}.
We have
\begin{equation*}
  0=t(\sD_1,W)-t(\sD_2,W)
=\int_{\oi^3}
W(x_1,x_2)W(x_2,x_3)\bigpar{1-W(x_1,x_3)} \dd x_1\dd
x_2\dd x_3
\end{equation*}
and \begin{equation*}
  0=t(\sD_3,W)
=\int_{\oi^3} W(x_1,x_2)W(x_2,x_3)W(x_3,x_1) \dd x_1\dd x_2\dd x_3.
\end{equation*}
Thus, \eqref{www} and \eqref{wc3} in the proof of
\refT{T1} hold. The proof of \refT{T1} actually used
the assumption that $R=\poow$ is a poset only to show
\eqref{www} and \eqref{wc3}; hence we may argue
exactly as for 
\ref{tcut5b}$\implies$\ref{tcut5c}.

\ref{tcut5b}$\implies$\ref{tcut5d}.
By \eqref{tfw},
\begin{equation*}
  t(\sD_3,W)=\P(1\precy2,\,2\precy3,\,3\precy1)
\end{equation*}
and
\begin{equation*}
  t(\sD_1,W)-t(\sD_2,W)=\P(1\precy2,\,2\precy3,\,1\not\precy3),
\end{equation*}
and both are 0 if $\poow$ \as{} is a poset.
\end{proof}

\begin{remark}\label{Rcut5}
  The implications
\ref{tcut5a}$\iff$\ref{tcut5b}$\implies$\ref{tcut5d} hold for
$W\in\cwsss$ for any \ps{} $\sss$. We do
not know whether that is true for the other implications, or whether
there might be measure theoretic complications.
\end{remark}

We prove a kernel version of \refL{Lcut0}.

\begin{lemma}
  \label{Lcut6}
If $W$ is a kernel on an \ops{} $\sfmuux$,
then there exists a kernel $W'$ on $\oibglx$, for some partial
order $\prec$ on $\oi$, such that $\dcut(W,W')=0$.
\end{lemma}

\begin{proof}
  By \refL{Lcut0}, there exists $W_1\in\cwoi$
  such that $\dcut(W,W_1)=0$.
If $F$ is any simple finite digraph, then
  \refL{Lcut2x} implies $t(F,W)=t(F,W_1)$. Since
  $\poow$ is a random infinite poset, \refL{Lcut5}
%\ref{tcut5b}$\implies$\ref{tcut5d}
  and \refR{Rcut5} show that 
$t(\sD_1,W_1)=t(\sD_1,W)=t(\sD_2,W)=t(\sD_2,W_1)$
  and 
$t(\sD_3,W_1)=t(\sD_3,W_1)=0$, and thus
  \refL{Lcut5} shows the existence of a kernel $W'$
  with $W'=W_1$ \aex{} and thus $\dcut(W,W')=\dcut(W,W_1)=0$.
\end{proof}

We define $\bcwp$ as
\begin{equation}
  \label{wp1}
\set{W:W\text{ is a kernel on some \ops{} $\sss$}},
\end{equation}
or
\begin{equation}
  \label{wp2}
\set{W:W\text{ is a kernel on $\oibglx$ for some $\prec$}},
\end{equation}
modulo the equivalence relation $\cong$; note that
\eqref{wp1} and \eqref{wp2} are equivalent by
\refL{Lcut6}. Thus $\bcwp$ is a subset of the metric
space $\bcw$, and
we equip $\bcwp$ with the inherited metric $\dcut$.

By \refL{Lcut2x}, the functionals $t(F,\cdot)$ are
well-defined and continuous on the quotient space $\bcw$.

\begin{lemma}
  \label{Lcut7}
$\bcwp=\set{\oW\in\bcw:
t(\sD_1,\oW)=t(\sD_2,\oW) \text{ and } t(\sD_3,\oW)=0}$.
\end{lemma}

\begin{proof}
  If $\oW\in\bcw$, we may by \refL{Lcut0}
  choose a representative in $\cwoi$, and the result then
  follows by \refL{Lcut5}.
\end{proof}

\begin{theorem}
  \label{Tcut2}
The metric space $(\bcwp,\dcut)$ is compact.
\end{theorem}

\begin{proof}
  $\bcwp$ is a closed subset of $\bcw$ by
  \refL{Lcut7} and the fact that the functionals
  $t(\sD_\ell,\cdot)$ are continuous on $\bcw$. Hence
  the result follows from \refT{Tcut1}.
\end{proof}

\section{Equivalence of kernels}\label{Suniqueness}

Suppose that $(\sss_1,\mu_1)$ and
$(\sss_2,\mu_2)$ are two probability spaces and that
$\gf:\sss_1\to\sss_2$ is a measure preserving map. If
$W:\sss_2\times\sss_2\to\bbR$, we let
$W^\gf:\sss_1\times\sss_1\to\bbR$ be the
function given by $W^\gf(x,y)=W\bigpar{\gf(x),\gf(y)}$.
If $\sss_2$ is an \ops{} with order
$\prec_2$ and $W$ is a kernel on $\sss_2$,
then we can define a partial order $\prec_1$ on
$\sss_1$ by $x\prec_1y\iff W^\gf(x,y)>0$; then 
$\sss_1$ is an \ops, $W^\gf$ is a (strict) kernel on
$\sss_1$, and $\gf:\sss_1\to\sss_2$ is
order preserving.
Furthermore, in this case, if $(X_i)_{i=1}^\infty$ are \iid{} points
in $\sss_1$, then $(\gf(X_i))_{i=1}^\infty$ are \iid{}
points in $\sss_2$, and it follows from \refD{Dpnw} that
\begin{equation}
  \label{pwgf}
P(n,W^\gf)\eqd\pnw
\qquad \text{for every $n\le\infty$};  
\end{equation}
hence
\refT{T1} implies that the kernels $W^\gf$ and
$W$ define the same poset limit $\Pi_W$. As in the case
of graph limits, see \cite{BCLSV1,BR,SJ209,BCL:unique},
this is not quite the only source of non-uniqueness of the
representing kernel $W$, but it is 'almost' so, in a sense made
precise below.

A \emph{Borel space} is a measurable space
$(\sss,\cF)$ that is isomorphic to a Borel subset of
$\oi$, see \eg{} \cite[Appendix A1]{Kallenberg} and \cite{Parthasarathy}. 
In fact,
a Borel space is either isomorphic to $(\oi,\cB)$ or it is
countable infinite or finite. Moreover, every Borel subset of a Polish
topological space (with the Borel $\gs$-field) is a Borel
space.
%\marginal{ref!}
A \emph{Borel probability space} is a probability space
$\sfmu$ such that $(\sss,\cF)$ is a Borel space.

We state a general equivalence theorem, which is the poset version of
\cite[Theorem 7.1]{SJ209} for graph limits. (This theorem in
\cite{SJ209} is for simplicity stated only for functions
defined on $\oibgl$, but it extends to arbitrary Borel
probability spaces in the same way as here.)
The parts \ref{TUtwin1} and \ref{TUtwin2} are
modelled after similar results for graph limits in
\cite{BCL:unique}.
(For graph limits, \cite{BCL:unique} also gives 
an equivalent
condition with $W_1=V^{\gf_1}$ and
$W_2=V^{\gf_2}$  for some $\gf_1$,
$\gf_2$ and $V$. We conjecture that a similar result is
true for poset limits too, but we have not yet investigated this.)

If $W$ is a kernel (or other function) $\sssq\to\oi$, where
$\sss$ is a probability space, we say following
\cite{BCL:unique} that
$x_1,x_2\in\sss$ are \emph{twins} if
$W(x_1,y)=W(x_2,y)$ and $W(y,x_1)=W(y,x_2)$ for
\aex{} $y\in\sss$. We say that $W$ is
\emph{almost twinfree} if there exists a null set
$N\subset\sss$ such that there are no twins
$x_1,x_2\in\sss\setminus N$ with $x_1\neq x_2$.

\begin{theorem}\label{TU}
Suppose that 
$W_1:\cS_1^2\to\oi$ and  $W_2:\cS_2^2\to\oi$ are two kernels
defined on two \ops{s}
$(\cS_1,\cF_1,\mu_1,\prec_1)$ and
$(\cS_2,\cF_2,\mu_2,\prec_2)$ such
that $(\cS_1,\mu_1)$ and $(\cS_2,\mu_2)$ are Borel spaces, 
and let\/
  $\Pi_1=\Pi_{W_1}$ and\/ $\Pi_2=\Pi_{W_2}$ be the
  corresponding poset limits in $\cpoo$.
Then the following are equivalent.
  \begin{romenumerate}
\item\label{TUgg}
$\Pi_1=\Pi_2$ in $\cpoo$.
\item\label{TUt}
$t(Q,\Pi_{1})=t(Q,\Pi_{2})$ for every poset $Q$.
\item\label{TUgoo}
The \exch{} random infinite posets $P(\infty,W_1)$ and $P(\infty,W_2)$
have the same distribution.
\item\label{TUgn}
The random posets $P(n,W_1)$ and $P(n,W_2)$
have the same distribution for every finite $n$.
\item\label{TUphi}
There exist measure preserving maps $\gf_j:\oi\to\cS_j$, $j=1,2$, such
that
$W_1^{\gf_1}=W_2^{\gf_2}$ \as, \ie{}
$W_1\bigpar{\gf_1(x),\gf_1(y)}=W_2\bigpar{\gf_2(x),\gf_2(y)}$ \aex{}
on $\oi^2$.
\item\label{TUpsi}
There exists a measurable mapping $\psi:\cS_1\times\oi\to\cS_2$ 
that maps $\mu_1\times\gl$ to $\mu_2$
such that
$W_1(x,y)=W_2\bigpar{\psi(x,t_1),\psi(y,t_2)}$ for \aex{}
$x,y\in\cS_1$ and  $t_1,t_2\in\oi$.
(Equivalently, if further
$\pi:\sss\=\sss_1\times\oi\to\sss_1$ is the
projection, then $W_1^\pi=W_2^\psi$ \as{} on
$\sss^ 2$.)
\item\label{TUcut}
$\dcut(W_1,W_2)=0$.
  \end{romenumerate}

If further $W_2$ is almost twinfree, then these are also equvalent to:
\begin{romenumerateq}
\item\label{TUtwin1}
There 
exists a measure preserving map $\gf:\sss_1\to\cS_2$ such
that
$W_1=W_2^{\gf}$ \as, \ie{}
$W_1{(x,y)}=W_2\bigpar{\gf(x),\gf(y)}$ \aex{}
on $\sss_1^2$.  
\end{romenumerateq}

If both $W_1$ and $W_2$ are almost twinfree, then these are also equvalent to:
\begin{romenumerateq}
\item\label{TUtwin2}
There exists a measure preserving map $\gf:\sss_1\to\cS_2$ such
that
$\gf$ is a bimeasurable bijection of
$\sss_1\setminus N_1$ onto $\sss_2\setminus
N_2$ for some null sets $N_1\subset\sss_1$ and
$N_2\subset\sss_2$, and 
$W_1=W_2^{\gf}$ \as, \ie{}
$W_1{(x,y)}=W_2\bigpar{\gf(x),\gf(y)}$ \aex{}
on $\sss_1^2$.  
(If further $(\sss_2,\mu_2)$ has no atoms, for example if
$\sss_2=\oi$, then we may take $N_1=N_2=\emptyset$.)
\end{romenumerateq}
\end{theorem}

\begin{proof}

\ref{TUgg}$\iff$\ref{TUt}. By our definition of $\cpoo\subset\cpq$ in
\refS{Slim}. 

\ref{TUgg}$\iff$\ref{TUgoo}. By \refT{TE}\ref{tex}.

\ref{TUgoo}$\iff$\ref{TUgn}. Obvious.

\ref{TUphi}$\implies$\ref{TUgoo},\ref{TUgn}. 
By \eqref{pwgf}, 
$P(n,W_1)\eqd P(n,W_1^{\gf_1})= P(n,W_2^{\gf_2})\eqd P(n,W_2)$.

\ref{TUpsi}$\implies$\ref{TUgoo},\ref{TUgn}. Similar.

\ref{TUgoo}$\implies$\ref{TUphi},\ref{TUpsi}. 
Consider first the case 
$(\cS_1,\mu_1)=(\cS_2,\mu_2)=(\oi,\gl)$.
In this case, \ref{TUphi} and \ref{TUpsi} follow,
as in the graph case in \cite{SJ209},
from Hoover's equivalence theorem for representations of
exchangeable arrays  in the version by 
\citet[Theorem 7.28]{Kallenberg:exch}; we refer to \cite[Proof of
  Theorem 7.1]{SJ209} for the details rather than copying them here.

For general $\sss_1$ and $\sss_2$, we first note
that since every Borel space is either finite, countably infinite or
(Borel) isomorphic to $\oi$, it is easily seen that there exist
measure preserving maps $\gam_j:\oi\to\cS_j$, $j=1,2$.
(Recall that $\oi$ is equipped with
the Lebesgue measure $\gl$ unless another measure is
explicitly given.)
Let $\tW_j\=W_j^{\gam_j}:\oi^2\to\oi$.
Then, by \eqref{pwgf},
$P(n,W_j)\eqd P(n,\tW_j)$ for $n\le\infty$, and thus
\ref{TUgoo} holds for $\tW_1$ and
$\tW_2$ defined on $\oi$. Hence, by the
special case just treated,
there exist measure preserving functions $\gf'_j:\oi\to\oi$
such that 
$\tW_1^{\gf'_1}=\tW_2^{\gf'_2}$ \aex{},
and thus \ref{TUphi} holds with $\gf_j\=\gam_j\circ\gf'_j$.

Similarly, by \ref{TUpsi} for $\tW_1$ and $\tW_2$,
there exists a measure preserving function $h:\oi^2\to\oi$ such that
$\tW_1(x,y)=\tW_2\bigpar{h(x,z_1),h(y,z_2)}$ for \aex{}
$x,y,z_1,z_2\in\oi$. Apply \refL{Linv} 
below with $(\cS,\mu)=(\cS_1,\mu_1)$ and $\gamma=\gam_1$. 
This yields $\ga:\cS_1\times\oi\to\oi$ that is measure preserving and with
$\gam_1(\ga(s,u))=s$ a.e.
Hence, for \aex{} $x,y\in\cS_1$ and $u_1,u_2,z_1,z_2\in\oi$,
\begin{equation*}
  \begin{split}
W_1(x,y)
&=
W_1\bigpar{\gam_1\circ \ga(x,u_1),\gam_1\circ \ga(y,u_2)}	
=
\tW_1\bigpar{ \ga(x,u_1), \ga(y,u_2)}	
\\&
=
\tW_2\bigpar{h(\ga(x,u_1),z_1),h(\ga(y,u_2),z_2)}
\\&
=
W_2\bigpar{\gam_2\circ h(\ga(x,u_1),z_1),\gam_2\circ h(\ga(y,u_2),z_2)}.
  \end{split}
\end{equation*}
Finally, let $\beta=(\beta_1,\beta_2)$ be a measure preserving map
$\oi\to\oi^2$, 
and define
$\psi(x,t)\=\gam_2\circ{h\bigpar{\ga(x,\beta_1(t)),\beta_2(t)}}$.

\ref{TUphi}$\implies$\ref{TUcut}. Obvious by \eqref{dcut0} and \refL{Lcut3}.

\ref{TUcut}$\implies$\ref{TUt}. 
By \refL{Lcut2x}.

\ref{TUpsi}$\implies$\ref{TUtwin1}. 
Since, for \aex{} $x,y,t_1,t_2,t_1'$,
\begin{equation*}
W_2\bigpar{\psi(x,t_1),\psi(y,t_2)}=
W_1(x,y)=W_2\bigpar{\psi(x,t_1'),\psi(y,t_2)}  
\end{equation*}
and
\begin{equation*}
W_2\bigpar{\psi(y,t_2),\psi(x,t_1)}=
W_1(y,x)=W_2\bigpar{\psi(y,t_2),\psi(x,t_1')}, 
\end{equation*}
and $\psi$ is  
measure preserving, it follows that for \aex{}
$x,t_1,t_1'$, $\psi(x,t_1)$ and $\psi(x,t_1')$
are twins for $W_2$. If $W_2$ is almost twin-free, with exceptional null
set $N$, then further
$\psi(x,t_1),\psi(x,t_1')\notin N$ for \aex{}
$x,t_1,t_1'$, since $\psi$ is \mpx, and consequently
$\psi(x,t_1)=\psi(x,t_1')$ for \aex{} $x,t_1,t_1'$.
It follows that we can choose a fixed $t_1'$ (almost every choice
will do) such that 
$\psi(x,t)=\psi(x,t_1')$ for \aex{} $x,t$. Define
$\gf(x)\=\psi(x,t_1')$.
Then $\psi(x,t)=\gf(x)$ for \aex{} $x,t$, which in particular
implies that $\gf$ is \mpx, and \ref{TUpsi}
yields $W_1(x,y)=W_2\bigpar{\gf(x),\gf(y)}$ a.e.

\ref{TUtwin1}$\implies$\ref{TUtwin2}. 
Let $N'\subset\sss_1$ be a null set such that if
$x\notin N'$, then $W_1(x,y)=W_2(\gf(x),\gf(y))$
for \aex{} $y\in\sss_1$.
Similarly, 
let $N''\subset\sss_1$ be a null set such that if
$x\notin N''$, then $W_1(y,x)=W_2(\gf(y),\gf(x))$
for \aex{} $y\in\sss_1$. If
$x,x'\in\sss_1\setminus(N'\cup N'')$ and
$\gf(x)=\gf(x')$, then $x$ and $x'$ are twins
for $W_1$. Consequently, if $W_1$ is almost twinfree
with exceptional null set $N$, then $\gf$ is injective
on $\sss_1\setminus N_1$ with $N_1\=N'\cup N''\cup N$. 
Since $\sss_1\setminus N_1$ and
$\sss_2$ are Borel spaces, the injective map
$\gf:\sss_1\setminus N_1\to\sss_2$ has measurable
range and is a bimeasurable bijection 
$\gf:\sss_1\setminus N_1\to\sss_2\setminus N_2$
for some measurable set $N_2\subset\sss_2$. Since
$\gf$ is measure preserving, $\mu_2(N_2)=0$. 
 
If $\sss_2$ has no atoms, we may take an uncountable null set
$N_2'\subset\sss_2\setminus N_2$. Let
$N_1'\=\gf\qw(N_2')$. Then $N_1\cup N_1'$ and
$N_2\cup N_2'$ are uncountable Borel spaces so there is a
bimeasurable bijection $\psi:N_1\cup N_1'\to N_2\cup
N_2'$. Redefine $\gf$ on $N_1\cup N_1'$ so that
$\gf=\psi$ there; then $\gf$ becomes a bijection $\sss_1\to\sss_2$.
\end{proof}

\begin{lemma}
  \label{Linv}
Suppose that $(\cS,\mu)$ is a Borel probability space and that
$\gamma:\oi\to\cS$ is a measure preserving function.
Then there exists a measure preserving function $\ga:\cS\times\oi\to\oi$
such that $\gamma\bigpar{\ga(s,y)}=s$ for $\mu\times \gl$-\aex{}
$(s,y)\in\cS\times\oi$. 
\end{lemma}

\begin{proof}
Let $\eta:\oi\to\oi$ and $\txi:\cS\to\cS$ be the identity maps
  $\eta(x)=x$, $\txi(s)=s$, and let $\xi=\gamma:\oi\to\cS$.
Then $(\xi,\eta)$ is a pair of random variables, 
defined on the probability space $(\oi,\gl)$,
with values in $\cS$ and $\oi$, respectively;
further, $\txi$ is a random variable defined on $(\cS,\mu)$ with
  $\txi\eqd\xi$.
By the transfer theorem \cite[Theorem 6.10]{Kallenberg}, there
  exists a measurable function $\ga:\cS\times\oi\to\oi$ such that 
if $\teta(s,y)\=\ga(\txi(s),y)=\ga(s,y)$, then $(\txi,\teta)$ is
  a pair of random 
  variables defined on $\cS\times\oi$ with
$(\txi,\teta)\eqd(\xi,\eta)$.
Since $\xi=\gamma(\eta)$, this implies $\txi=\gamma(\teta)$ \as, 
and thus $s=\txi(s)=\gamma\bigpar{\ga(s,y)}$ \as
\end{proof}

\section{More on the cut metric}\label{Scut2}

\begin{theorem}
    \label{Tcutw}
Let $W$ and $W_1,W_2,\dots$ be kernels on \ops{s}
$\sss,\sss_1,\sss_2,\dots$. Then, as \ntoo,
$\Pi_{W_n}\to\piw\iff\dcut(W_n,W)\to0$.
In other words, 
the mapping $W\mapsto\piw$ is a
homeomorphism of $(\bcwp,\dcut)$ onto $\cpoo$.
\end{theorem}

\begin{proof}
  The mapping $W\mapsto\piw\in\cpoo$ is well-defined and
  continuous on $\bcwp$ by \refL{Lcut2x} and the
  construction of $\cpoo$ (see \refT{T1ox});
  further, the mapping is surjective by \refT{T1} and it is
  injective by \refT{TU} (\ref{TUgg}$\implies$\ref{TUcut}), using the
  definition \eqref{wp2}.
Since $\bcwp$ is compact by \refT{Tcut2},
  the mapping is thus a homeomorphism.
\end{proof}

\begin{proof}[Proof of \refT{Tcutp}]
Let $W_n=W_{P_n}$. Thus $\Pi_{P_n}\=\Pi_{W_n}$ and
$t(Q,P_n)=t(Q,\Pi_{P_n})=t(Q,\Pi_{W_n})$ for every
  $Q\in\cP$ by \refE{EP}. 
It follows from Theorems \refand{T1old}{T1ox} that
$P_n\to\Pi\iff \Pi_{W_n}\to\Pi$, and the
result follows from \refT{Tcutw}.
\end{proof}

\section{Further examples}\label{Sex}

\begin{example}
  \label{Etotal}
For each finite $n$, all totally ordered sets with $n$ elements are
isomorphic, and there is thus a unique unlabelled totally ordered
poset in $\cP_n$ which we denote by $T_n$. Let $\sfmuu$ be a totally
ordered set with a continuous probability measure $\mu$ (\ie, a
probability measure such that $\mu\set x=0$ for every $x\in\sss$), and
let $W(x,y)=\ett{x<y}$ as in \refE{EW1}. Since $\mu$ is continuous,
the random points $X_i$ in \refD{Dpnw} are (\as) distinct, and thus,
see \refE{EW1}, $\pnw$ is isomorphic to a subset of $\sss$ and thus
totally ordered. In other words, $\pnw=T_n$ as unlabelled posets. (As
labelled posets, $\pnw\eqd\widehat T_n$, which is obtained by applying a
random permutation to $\nn$ with the usual order.) By \refT{T1}\ref{T1a}, thus
$T_n\to\piw$, which shows that $\piw$ does not depend on the choices
of $\sss$ and $\mu$. We write $\pit$ for this poset limit and have
thus shown that there exists a (unique) poset limit $\pit\in\cpoo$
such that
$T_n\to\pit$ and $P(n,\pit)=T_n$ for all finite $n$.
We may call $\pit$ the \emph{total poset limit}.

It is convenient to choose $\sss$ as \oilm; we then see that
$P(\infty,\pit)$ is the random infinite total order defined by a
sequence of \iid{} random points in $\oi$ with the standard order.

Note that $\mu$ has to be continuous in this example; otherwise (\ie, if
$\mu$ has an atom), there will (\as) be repetitions in
$X_1,X_2,\dots$ and thus incomparable points in $\poow$ (and with
positive probability in $\pnw$ for finite $n\ge2$); hence
$\poopiw=\poow\not\eqd\poox{\pit}$ and
$\piw\neq\pit$ by \refT{TE}\ref{tex}.
In particular, $\pip$ defined in \refE{EP} for a
finite totally ordered set $P=T_m$ does \emph{not} equal
$\pit$. (Although, as a consequence of \eqref{tp}, $\Pi_{T_m}\to\pit$
in $\cpoo$ as $m\to\infty$.) 
\end{example}

\begin{example}
  \label{E0}
The other extreme is the poset where $x\not<y$ for all $x,y$; we call
these posets trivial, and let $E_n$ denote the (unique) unlabelled
trivial poset with $|E_n|=n$.
Then, trivially, $t(Q,E_n)=0$ for every finite poset $Q$ that is not
itself trivial, while $t(E_m,E_n)=1$ for all $m$ and $n$. Consequently
the sequence $(E_n)$ converges, and the limit is a poset limit
$\pie\in\cpoo$ with
\begin{equation}\label{e0}
  t(Q,\pie)=
  \begin{cases}
	1,& Q= E_m \text{ for some $m$},\\
0,& \text{otherwise}.
  \end{cases}
\end{equation}
Taking $P=E_n$ in \refE{EP}, we see further by \eqref{tp}
that $\Pi_{E_n}=\pie$ for every $n$.
Trivially, $\widehat E_n=E_n$, and by \eqref{tqpi}
$P(\infty,\pie)$ is the trivial infinite poset on $\bbN$.
Similarly,
$P(n,\pie)=P(\infty,\pie)\rest n$ is trivial, so $P(n,\pie)=E_n$.

Note also that if $\sss$ is any \ops{} and $W=0$, which always is a
kernel, then 
$\pnw$ is trivial for all $n<\infty$, and by
\refT{T1}\ref{T1a} or \ref{T1b},
$\piw=\pie$. (This explains our notation $\pie$.)
\end{example}

\begin{example}
  \label{E2dim}
Let $\sss=\oi^2$ with \lm{} and the product order
$(x_1,x_2)<(y_1,y_2)$ if $x_1<y_1$ and $x_2<y_2$. Again, let
$W(x,y)=\ett{x<y}$ as in \refE{EW1}. Then $\pnw$ is the poset defined
by $n$ random points in $\oi^2$, which also can be described as the
intersection of two independent random total orders on $\nn$.
\end{example}

\begin{example}
Let $\gnp$ denote the random graph with $n$ vertices \set{1,\dots,n}
where each possible edge $ij$ appears with probability $p$,
independently of all other edges. We make \gnp{} into a (random) poset
by directing each edge from the smaller endpoint to the larger, and
then taking the transitive closure. In other words, $i\prec j$ in
\gnp{} if and only if there is an increasing path
$i=i_1,i_2,\dots,i_n=j$ in \gnp. We use \gnp{} to denote this random
poset too.

It can be shown, see \cite{PittelT} and the references therein, that if
$p\to0$ and $(j-i)/(\frac1p\log\frac1p)\to c$, then 
$\P(i\prec j)\to0$ if $c<1$ and 
$\P(i\prec j)\to1$ if $c>1$.
Assume now that $n\to\infty$ and $p\to0$ such that $pn/\log n\to a\in
[0,\infty]$. It then follows easily that for every finite poset $Q$,
using \eqref{tfw1} and \eqref{t1b},
\begin{equation*}
  \E t(Q,\gnp)\to t(Q,W_a)=t(Q,\Pi_{W_a}),
\end{equation*}
where $W_a$ is the kernel on $\oibgl$ given by
$W_a(x,y)\=\ett{y-x>a\qw}$.
(In particular, $W_a=0$ if $a\le1$.)
By \refT{T2}, thus $\gnp\dto\Pi_{W_a}$;
since $\Pi_{W_a}$ is non-random, this means
$\gnp\pto\Pi_{W_a}$.

In particular, if $a\le1$, then $\gnp\pto\pie$, see \refE{E0}.
The other extreme is $a=\infty$; then $W_a(x,y)=\ett{y>x}$ on the
totally ordered set $\oi$, so $\gnp\pto\pit$, see \refE{Etotal}.
\end{example}

\begin{example}
Let $\sss=\set{(x,y):0\le x\le y\le 1}$ with the partial order
$(x_1,y_1)\prec(x_2,y_2)$ if $y_1<x_2$. We can interpret $\sss$ as the
set of closed intervals in \oi, with $I_1\prec I_2$ if $I_1$ lies
entirely to the left of $I_2$. Any probability measure $\mu$ on $\sss$
thus defines a distribution of random intervals, and the kernel 
$W(\mathbf x,\mathbf y)\=\ett{\mathbf x\prec\mathbf y}$ as in
\refE{EW1} yields random posets $\pnw$, and a poset limit $\Pi$.

We note that although it is natural to represent $\Pi$ by the kernel
$W$ on $(\sss,\mu)$, $\Pi$ can also be represented by a kernel on
$\oibgl$. (Thus 
\refP{P1} has a positive answer in this case.)
To see this, we construct a \mpx{} map $\gf:(\oi,\gl)\to(\sss,\mu)$
such that $\gf(s)\prec\gf(t)\implies s<t$; then $W^\gf$ is a kernel on
$\oi$ that represents $\Pi$.
We may construct $\gf$ by first partitioning $\sss$ into
$\sss_0\=\set{(x,y):x\le y<\xfrac12}$,
$\sss_{01}\=\set{(x,y):x<\xfrac12\le y}$ and
$\sss_1\=\set{(x,y):\xfrac12\le x\le y}$, and a corresponding partitioning
of $\oi$ into $I_0\=[0,\mu(\sss_0))$,
$I_{01}\=[\mu(\sss_0),1-\mu(\sss_1))$ and
$I_{1}\=[1-\mu(\sss_1),1]$.
Noting that all elements of
$\sss_{01}$ are incomparable, we define $\gf$ on $I_{01}$ as any \mpx{}
map $I_{01}\to\sss_{01}$. We then continue recursively and define
$\gf:I_0\to\sss_0$ and $I_1\to\sss_1$ by partitioning $\sss_0$ and
$\sss_1$ into three parts each, and so on. (In the $k$th stage, the
partitioning is according to the $k$th binary digit of $x$ and $y$.)
Let $\gD\=\set{(x,x)}$ be the diagonal in $\sss$. If $\mu(\gD)=0$,
then the recursive procedure just described defines $\gf$ at least
\aex{} on $\oi$. If $\mu(\gD)>0$, there will remain a Cantor like subset
of $\oi$ of measure $\mu(\gD)$; the construction then is completed by
mapping this set to $\gD$ by an increasing \mpx{} map.
\end{example}

\section{Poset limits as digraph limits}\label{SD}

As said repeatedly, we can regard posets as digraphs, which 
yields an inclusion mapping $\cP\to\cD$.
We saw in \refS{Slim} that this mapping extends to a (unique) continuous
inclusion mapping $\cpq\to\cdq$; we may thus regard 
$\cpq$ as a compact subset of $\cdq$, with
$\cpoo$ a compact subset of $\cdoo$. 
We can now characterize the subset $\cpoo$ of $\cdoo$ in
several ways.

We first recall that, as shown in \citet{SJ209}, the digraph limits in
$\cdoo$ can be represented by
quintuples
$\bW=(W_{00},W_{01},W_{10},W_{11},\wo)$ where  
$\wab:\oi^2\to\oi$ and $\wo:\oi\to\setoi$ are measurable functions 
such that $\sum_{\ga,\gb=0}^1\wab(x,y)=1$ and $\wab(x,y)=W_{\gb\ga}(y,x)$
for $\ga,\gb\in\setoi$ and $x,y\in\oi$.
Let $\WW$ be the set of all such quintuples.
For $\bW\in\WW$, we define a random infinite digraph
$\gbwoo$
by specifying its edge indicators $I_{ij}$
as follows (\cf{} \refD{Dpnw}): we
first choose a sequence $X_1,X_2,\dots$ of \iid{} random variables
uniformly distributed on $\oi$, and then, given this sequence,
let $I_{ii}=\wo(X_i)$ and 
for each pair $(i,j)$ with $i<j$ choose $I_{ij}$ and $I_{ji}$ at random
such that
\begin{equation}\label{dirw}
 \P(I_{ij}=\ga \text{ and } I_{ji}=\gb)=\wab(X_i,X_j),
\qquad \ga,\gb\in\setoi;
\end{equation}
this is done independently for
all pairs $(i,j)$ with $i<j$ (conditionally given \set{X_k}).
%In other words, for every $i$ we draw a loop at $i$ if $\wo(X_i)=1$ 
%and
%for each pair $(i,j)$ with $i<j$ we 
%draw edges $ij$ and $ji$ at random such that \eqref{dirw} holds.
%
%Further,  $\gbwn\=\gbwoo\rest n$, which is
%obtained by the same construction with a finite sequence $X_1,\dots,X_n$.
The infinite random digraph $\gbwoo$ is \exch, and it is
shown in \cite{SJ209}, by digraph analogues of Theorems
\refand{TE2}{TE} above, that its distribution is an extreme
point in the set of \exch{} distributions and that it
corresponds to a digraph limit $\ubw$; for example, 
$\gbwn\=\gbwoo\rest n\to\ubw$ in $\cdq$ \as

\begin{theorem}
  \label{TPQ}
Let $\gG\in\cdoo$ be a digraph limit. Then the following are
equivalent.
\begin{romenumerate}
  \item $\gG\in\cpoo$, \ie, $\gG$ is a poset limit.
\item
$\tind(F,\gG)=0$ for every finite digraph $F$ that is not a poset.
\item
$\tind(\sC_1,\gG)=\tind(\sC_2,\gG)=\tind(\sC_3,\gG)=\tind(\sP_2,\gG)=0$.
\item
If $\bW=(W_{00},W_{01},W_{10},W_{11},\wo)$ is some (any) quintuplet
representing $\gG$, then $w=0$ a.e., $W_{11}=0$ \aex, and 
$\set{(x,y,z):W_{10}(x,y)>0 \text{ and }W_{10}(y,z)>0\text{ 
and }W_{10}(x,z)<1}$ is a null set in $\oi^3$.
\item
There exists a quintuplet
$\bW=(W_{00},W_{01},W_{10},W_{11},\wo)$ 
representing $\gG$ with  $w=0$, $W_{11}(x,y)=0$, $W_{10}(x,x)=0$, and 
$W_{10}$ satisfying \eqref{w2}.
\end{romenumerate}
\end{theorem}
  
\begin{proof}
  (i)$\implies$(ii).
If $P$ is a poset regarded as a digraph, then every induced
subgraph is a poset. Thus, if
$F\in\cD\setminus\cP$, then
$\tind(F,P)=0$ for all $P\in\cP$, and by
continuity, $\tind(F,\gG)=0$ for all
$\gG\in\cpoo$ too.
%Let $P_n\in\cP$ with $P_n\to\gG$. Then, for
%every finite digraph $F$ that is not a poset,

(ii)$\implies$(iii). Trivial.

(iii)$\implies$(iv). Let 
$\bW=(W_{00},W_{01},W_{10},W_{11},\wo)$ represent the digraph limit
$\gG$ as above. Then, by the digraph version of
\eqref{tqpi}, for once regarding
$\sC_1,\sC_2,\sC_3,\sP_2$ as labelled digraphs (in the obvious way),
\begin{equation*}
  0=t(\sC_1,\gG)=\P(\sC_1\subset  \gbwoo)=\P(I_{11}=1)=\E w(X_1).
\end{equation*}
Thus $w=0$ a.e.
Similarly,
\begin{equation*}
  0=t(\sC_2,\gG)=\P(\sC_2\subset \gbwoo)=\P(I_{12}=I_{21}=1)=\E W_{11}(X_1),
\end{equation*}
and thus $W_{11}=0$ a.e.
Finally, using $W_{11}=0$, 
\begin{equation*}
  \begin{split}
  0&=t(\sC_3,\gG)+t(\sP_2,\gG)
=\P(\sC_3\subset \gbwoo)+\P(\sP_2\subset \gbwoo)
\\&=\P(I_{12}=I_{23}=1,\, I_{31}=0)
\\&
=\E W_{10}(X_1,X_2)	W_{10}(X_2,X_3)(1-W_{10}(X_3,X_1)).
  \end{split}
\end{equation*}

(iv)$\implies$(i). 
By the calculations in the preceding step, 
\begin{multline*}
  \P(\sC_1\subset \gbwoo)
=\P(\sC_2\subset \gbwoo)
\\
=\P(\sC_3\subset \gbwoo)+\P(\sP_2\subset \gbwoo)
=0.	
\end{multline*}
By exchangeability, $\gbwoo$ thus \as{} does not have
any induced subgraph $\sC_1$, $\sC_2$,
$\sC_3$ or $\sP_2$, and thus \refL{L1}
shows that $\gbwoo$ and
its induced subgraphs $\gbwn$ are posets a.s.
Since $\gG=\lim\gbwn$ \as, $\gG$ is a limit of posets.

(i)$\implies$(v). 
By Theorems \refand{T1}{T1+}\ref{t1+oi}, we
can represent $\gG$ regarded as a poset limit by a kernel $W$ on
$\oib$ (with some partial order $\prec$). We define
$W_{10}(x,y)=W(x,y)$,  $W_{01}(x,y)=W(y,x)$, 
$W_{11}(x,y)=0$, 
$W_{00}(x,y)=1-W(x,y)-W(y,x)$
and $w(x)=0$. 
(Alternatively, we can show (iv)$\implies$(v) by modifying $\bW$ on a
null set similarly to the proof of \refT{T1}.)

(v)$\implies$(iv). Trivial.
\end{proof}

\section{Further comments}\label{Sfurther}

One might ask for topological properties of the compact metric space
$\cpoo$. We only give one simple result here.

\begin{theorem}\label{Tcontr}
  $\cpoo$ is a contractible topological space, and in particular
  connected and simply connected.
\end{theorem}

\begin{proof}
  To be contractible means that there is a homotopy between the
  identity map $\cpoo\to\cpoo$ and a constant map, \ie, a continuous
  map $\Psi:\cpoo\times\oi\to\cpoo$ such that $\Psi(\Pi,0)=\Pi$ and
  $\Psi(\Pi,1)=\Pi'$ for all $\Pi$ and some fixed $\Pi'$ in $\cpoo$.
We construct such a map with $\Pi'=\pie$ as follows.

Given $\Pi\in\cpoo$, choose a representing kernel $W$ on an \ops{}
$\sfmuu$. Define $\sssx\=\sss\cup\set *$ (with $*\notin\sss$), extend
$<$ to $\sssx$ in any way (\eg, with $*$ incomparable to
every $x\in\sss$), and define, for $p\in\oi$, $\mup\set*=1-p$
and $\mup(A)=p\mu(A)$ for $A\subseteq\sss$; finally, extend $W$ to
$\sssx$ by $W(*,x)=W(x,*)=W(*,*)=0$ for $x\in\sss$.
Let $\pipp\in\cpoo$ be the poset limit defined be the extended kernel
$W$ on $(\sssx,\mup)$. 

For a poset $Q$, let 
\begin{equation*}
  \np{Q}\=|\set{x\in Q:x<y \text{ or } y<x \text{ for some $y\in Q$}}|,
\end{equation*}
%$\np{Q}\=|\set{x\in Q:x<y \text{ or } y<x \text{ for some $y$}}|$, 
the
number of elements of $Q$ that are comparable to at least one other
element.
Then, as a consequence of \refT{T1}\ref{T1b}, for every finite poset $Q$,
\begin{equation}
  \label{homo}
t(Q,\pipp)=p^{\np{Q}}t(Q,\Pi).
\end{equation}
In particular, this shows that $\pipp$ depends on $\Pi$ and $p$ only,
and not on the choice of the kernel $W$. Furthermore, $\Pi_{(1)}=\Pi$,
while, by \eqref{homo} and \eqref{e0}, $\Pi_{(0)}=\pie$ defined in
\refE{E0}, for every $\Pi\in\cpoo$. Moreover, \eqref{homo} shows that
the map $(\Pi,s)\mapsto\Pi_{(s)}$ is a continuous map
$\cpoo\times\oi\to\cpoo$. 
Consequently, $\Psi(\Pi,s)\=\Pi_{(1-s)}$ defines the desired homotopy.
\end{proof}

The poset limit $\pipp$ in the proof can be regarded as a thinning of
$\Pi$. The corresponding \erip{} $\P(\infty,\pipp)$ is obtained from
$\poopi$ by randomly selecting elements with probability $1-p$ each,
independently, and making them uncomparable to everything.

\begin{ack}
This work was stimulated by helpful discussions with 
Graham Brightwell and Malwina Luczak during the programme
``Combinatorics and Statistical Mechanics'' at the Isaac Newton
Institute, Cambridge, 2008, where SJ was supported by a Microsoft
fellowship.
Parts of this work was done at Institut Mittag-Leffler, Djursholm, 2009.
\end{ack}

\newcommand\AAP{\emph{Adv. Appl. Probab.} }
\newcommand\JAP{\emph{J. Appl. Probab.} }
\newcommand\JAMS{\emph{J. \AMS} }
\newcommand\MAMS{\emph{Memoirs \AMS} }
\newcommand\PAMS{\emph{Proc. \AMS} }
\newcommand\TAMS{\emph{Trans. \AMS} }
\newcommand\AnnMS{\emph{Ann. Math. Statist.} }
\newcommand\AnnPr{\emph{Ann. Probab.} }
\newcommand\CPC{\emph{Combin. Probab. Comput.} }
\newcommand\JMAA{\emph{J. Math. Anal. Appl.} }
\newcommand\RSA{\emph{Random Struct. Alg.} }
\newcommand\ZW{\emph{Z. Wahrsch. Verw. Gebiete} }
\newcommand\DMTCS{\jour{Discr. Math. Theor. Comput. Sci.} }

\newcommand\AMS{Amer. Math. Soc.}
\newcommand\Springer{Springer}
\newcommand\Wiley{Wiley}

\newcommand\vol{\textbf}
\newcommand\jour{\emph}
\newcommand\book{\emph}
\newcommand\inbook{\emph}
\def\no#1#2,{\unskip#2, no. #1,} %(typeset after year) 
\newcommand\toappear{\unskip, to appear}

\newcommand\webcite[1]{%\hfil  %???
   %\penalty0 %???
\texttt{\def~{{\tiny$\sim$}}#1}\hfill\hfill}
\newcommand\webcitesvante{\webcite{http://www.math.uu.se/~svante/papers/}}
\newcommand\arxiv[1]{\webcite{arXiv:#1.}}

\def\nobibitem#1\par{}

\end{document}